\documentclass[a4paper,12pt]{article}

\usepackage[centertags]{amsmath}
\usepackage{amsfonts}
\usepackage{amssymb}
\usepackage{amsthm}
\usepackage{dsfont}
\usepackage{tikz, subfigure}
\usepackage{verbatim}
\usepackage{graphicx}

\newtheorem{theorem}{Theorem}[section]
\newtheorem{lemma}{Lemma}[section]

\newtheorem{prop}{Proposition}[section]
\newtheorem{definition}{Definition}[section]

\newtheorem{conjecture}{Conjecture}

\newtheorem{condition}{Condition}[section]
\newtheorem{corollary}{Corollary}[section]
\newtheorem{proposition}{Proposition}[section]

\newcommand{\gl}{{\lambda}}
\newcommand{\gk}{{\kappa}}

\newcommand{\gb}{{\beta}}

\newcommand{\e}{{\varepsilon}}

\begin{document}

\title{Measures on the Spectra of Algebraic Integers}
\author{Alex Batsis and Tom Kempton\footnote{The University of Manchester, Oxford Road, Manchester, M13 9PL, United Kingdom.}}
\maketitle

\begin{abstract}{\noindent 
Given a real number $\beta>1$, the spectrum of $\beta$ is a well studied dynamical object. In this article we show the existence of a certain measure on the spectrum of $\beta$ related to the distribution of random polynomials in $\beta$, and discuss the local structure of this measure. We also make links with the question of the Hausdorff dimension of the corresponding Bernoulli Convolution.}
\end{abstract}

\section{Introduction}
Given a real number $\beta>1$ and an alphabet $\mathcal A$, the spectrum 
\[
X_{\mathcal A}(\beta):=\left\{\sum_{i=1}^{n} c_i\beta^{n-i}: n\in\mathbb N, c_i\in\mathcal A\right\}
\]
has been the focus of much attention. In particular, when $\mathcal A=\{0,\cdots, \lfloor \beta \rfloor\}$ then it is known  that $X_{\mathcal A}(\beta)$ is uniformly discrete if and only if $\beta$ is a Pisot number (i.e. an algebraic number, all of whose Galois conjugates have modulus strictly less than one) \cite{ AkiyamaKomornik, Bugeaud96, FengTopology, GarsiaEntropy}. Additionally, $X_{\mathcal A}(\beta)$ is relatively dense in this setting, making the sets $X_{\mathcal A}(\beta)$ Delone sets (uniformly discrete, relatively dense). Delone sets give useful mathematical models for quasicrystals and so the above construction gives a number-theoretic construction of important physical objects.

Much progress has been made on giving dynamical descriptions of sets $X_{\mathcal A}(\beta)$ \cite{EJK, FengWen, HareMasakovaVavra}. If $\beta$ is a Pisot number then $X_{\mathcal A}(\beta)$ can be generated by a substitution system \cite{FengWen}. Moreover, for Pisot $\beta$ there is a naturally related cut and project set which contains $X_{\mathcal A}(\beta)$. In all known examples of Pisot $\beta$ with $\mathcal A\subset \mathbb Z$ the set $X_{\mathcal A}(\beta)$ coincides with this cut and project set, but the question of whether these sets always coincide remains open, and there are some examples with a complex alphabet for which the cut and project set contains finitely many extra points which are not in $X_{\mathcal A}(\beta)$ \cite{HareMasakovaVavra}. A generalisation of this cut and project structure to general hyperbolic algebraic integers is given in section \ref{GeneralSec}.


We are interested in measures on the sets $X_{\{-1,0,1\}}(\beta)$. In particular, we are interested in what one can say about the measures $\mu_n$ given by
\[
\mu_n(x)=\dfrac{1}{4^n}\mathcal N_n(x)
\]
where
\[
\mathcal N_n(x)=\#\{a_1\cdots a_n, b_1\cdots b_n \in \{0,1\}^n: \sum_{i=1}^n (a_i-b_i)\beta^{n-i}=x\}.
\]
The measure $\mu_n$ is the distribution of the set of differences
\[
\sum_{i=1}^n a_i\beta^{n-i}-\sum_{i=1}^n b_i\beta^{n-i}
\]
where each $a_i, b_i$ is picked from $\{0,1\}$ according to the $(\frac{1}{2},\frac{1}{2})$ Bernoulli measure.\footnote{There has been a lot of recent research into a different class of measures (Patterson measures) on cut and project sets. These are related to diffraction on quasicrystals, where they play the role of the intensity of the Bragg peak \cite{Lenz08, RichardStrungaru}. Loosely speaking, the difference between the class of measures that we study and Patterson measures is that our measures incorporate information on the number of different codings $a_1\cdots a_n$ for which $\sum_{i=1}^n a_i\beta^{n-i}=x$, whereas Patterson measures do not. The analogue of $\mu_n(x)$ for the Patterson measure would be (more or less) \[\gamma_n(x)=\#\{(y,z)\in (X_{\{0,1\}}(\beta))^2: y-z=x \}.\] This difference is crucial for our applications.} We focus on the case that $\beta$ is an algebraic integer and a root of a \{-1,0,1\} polynomial but does not have any Galois conjugates of absolute value one, we call such $\beta$ hyperbolic.

Broadly, we are interested in the question of whether the measures $\mu_n$, appropriately rescaled, have a limit $\mu$ as $n$ tends to infinity, and whether that limit has any `local structure' analagous to that of the set $X_{\mathcal A}(\beta)$. Assuming some technical (but checkable) conditions, our results hold for general hyperbolic $\beta$, but all of the ideas behind our proofs are present in the golden mean case, which is notationally much simpler, and for this reason we prove our results first for the golden mean. The golden mean also has the advantage that the higher dimensional objects which we construct are only two dimensional, and so can be more easily visualised. 

Our main theorems are the following.
\begin{theorem}\label{Thm1}
Let $\beta$ be hyperbolic. Then there exists a real number $\lambda>1$, such that for all $x\in X(\beta)$ the limit measure $\mu$ given by
\[
\mu(x):=\lim_{n\to\infty}\frac{1}{\lambda^n}\mathcal N_n(x)
\]
exists and has $\mu(x)\in(0,\infty)$ for $x\in X(\beta)$. Furthermore, the measure $\mu$ has infinite total mass.  
\end{theorem}
In the case that $\beta$ has other Galois conjugates of absolute value larger than one, we prove this theorem by lifting to a measure $\bar{\mu}$ supported on a higher dimensional Delone set, whose projection onto the first coordinate gives $\mu$.

Our second theorem gives an explicit way to calculate $\mu(x)$ using any code of $x$.
\begin{theorem}\label{Thm2}
Let $\beta$ be hyperbolic. There exist a natural number $k$, a $1\times k$ vector $W$, and three $k\times k$ matrices $M_{-1}, M_0$ and $M_1$ such that for any $x\in X(\beta)$ and $c_1\cdots c_n\in\{-1,0,1\}^n$ with $x=\sum_{i=1}^n c_i\beta^{n-i}$, 
\[
\mu(x)=\frac{1}{\lambda^n}(WM_{c_1}\cdots M_{c_n})_1.
\]
Here $(WM_{c_1}\cdots M_{c_n})_1$ denotes the first entry of the row vector $WM_{c_1}\cdots M_{c_n}$.
\end{theorem}

In fact the vector $WM_{c_1}\cdots M_{c_n}$ also holds information on the values of $\mu(y)$ for other values of $y\in X(\beta)$. There is a set of translations $d_1,\cdots, d_k\in\mathbb R$, with $d_1=0$, such that, for $x=\sum_{i=1}^n c_i\beta^{n-i}$,  
\[
\frac{\mu(x+d_i)}{\mu(x)}=\frac{(WM_{c_1}\cdots M_{c_n})_i}{(WM_{c_1}\cdots M_{c_n})_1}.
\]
This suggests that one may be able to use a dynamical system to move through the measure $\mu$ to calculate its values at different points. We can do this, but we need first to replace the dependence of $\mu(x)$ on the coding of $x$ with a dependence on the position of a point $x_c$ corresponding to $x$ in the `contracting space'. To describe this, we must first describe a geometric construction related to $\beta$-expansions in algebraic bases.

Let $\beta$ have Galois conjugates $\beta_2\cdots \beta_d$ of absolute value larger than one and Galois conjugates $\beta_{d+1}\cdots \beta_{d+s}$ of absolute value smaller than one. Define the contracting space $\mathbb K_c$ by $\mathbb K_c=\mathbb F_{d+1}\times \mathbb F_{d+2}\times \cdots \times \mathbb F_{d+s}$ where $\mathbb F_k=\mathbb R$ if $\beta_k\in\mathbb R$, $\mathbb F_k=\mathbb C$ if $\beta_k\in\mathbb C\backslash\mathbb{R}$. Then, for $i\in \{-1,0,1\}$ define the contraction $S_i$ on $\mathbb K_c$ by
\[
S_i(x_{d+1},\cdots, x_{d+s})=(\beta_{d+1}x_{d+1}+i,\cdots, \beta_{d+s}x_{d+s}+i). 
\]
The maps $\{S_{-1},S_0,S_1\}$ form an iterated function system on $\mathbb K_c$ with an attractor that we denote $\mathcal R$. This is a standard construction in numeration/tiling theory, although it is more usual to consider a sub-IFS using only those codes which correspond to greedy $\beta$-expansions \cite{AkiyamaTiling}. To each point $x=\sum_{i=1}^n c_i\beta^{n-i}$ there exists a corresponding point in the contracting space:
\[
x_c=\sum_{i=1}^n c_i(\beta_{d+1}^{n-i}, \beta_{d+2}^{n-i},\cdots, \beta_{d+s}^{n-i})=S_{c_n}\circ \cdots S_{c_1}(0)\in \mathcal R.
\]
It is important to stress that  the point $x_c$ corresponding to $x$ is independent of the coding $c_1,\cdots,c_n$ of $x$, this holds since $\beta_{d+1}\cdots \beta_{d+s}$ are Galois conjugates of $\beta$.

\begin{theorem}\label{Thm3}
Assume that Condition \ref{condition 1} holds. There exists a set $\Delta=(v_1,\cdots v_k)$ of translations such that for any $j\in\{1\cdots k\}$ there is a function $f_j:\mathcal R\to \mathbb R$ such that for any $x\in X(\beta)$ with $x+v_j$ also in $X(\beta)$ we have
\[
\ln\left(\frac{\mu(x+v_j)}{\mu(x)}\right)=f_j(x_c).
\]
Furthermore any $x\in X(\beta)$ can be reached from $0$ by applying a finite number of translations from $\Delta$. There exists a word $w$ and constants $C_1>0$, $C_2\in(0,1)$ such that for any $a_1\cdots a_n\in\{-1,0,1\}^n$ which contains $r$ non-overlapping copies of the word $w$, $f_j$ varies by at most $C_1C_2^{r-1}$ on $S_{a_1}\circ \cdots \circ S_{a_n}(\mathcal R)$.
\end{theorem}

The final condition on the variation of $f_j$ gives rise to the following continuity properties of $f_j$.
\begin{enumerate}
\item {\bf Continuity almost everywhere:} For any fully supported ergodic measure $\nu$ on $\mathcal R$, each $f_j$ is continuous $\nu$-almost everywhere
\item {\bf Continuity at most lattice points:} For any fully supported measure $m$ on $\{-1,0,1\}$ and any $\epsilon>0$ there exists $n\in\mathbb{N}$ and $D\subseteq\{-1,0,1\}^n$ such that $m^n(D)> 1-\e$ and
  
  \[
  |f_j(x)-f_j(y)|<\e
  \]
  for all $x,y\in X(\beta)$ with $x_c, y_c\in S_{a_1}\circ\cdots \circ S_{a_n}(\mathcal R)$ for any $a_1\cdots a_n\in D$.
  

\end{enumerate}
These latter two continuity properties follow since $\nu$ almost every sequence contains infinitely many copies of the word $w$, and that for any $r$ and any $\epsilon>0$ there exists $n$ such that a proportion at least $1-\epsilon$ of $\{-1,0,1\}$ words of length $n$ contain $r$ non-overlapping occurences of $w$.

We use this theorem extensively in our follow up article. For now, we limit our application of this theorem to the golden mean case, where we show that the values of $\mu(x)$ can be obtained via a cocycle over an interval exchange transformation on $\mathcal R=(-\phi^2,\phi^2)$, see Theorem \ref{GMCocycle}.


In Section \ref{BCSection} we describe some links with the dimension theory of Bernoulli convolutions, which allows us to state some new conjectures about Bernoulli convolutions. In Section \ref{GM} we prove Theorems \ref{Thm1}, \ref{Thm2} and \ref{Thm3} in the special case that $\beta$ is the golden mean. Finally in Section \ref{GeneralSec} we prove these theorems for the general case of hyperbolic $\beta$.

\section{Links to the Dimension Theory of Bernoulli Convolutions}\label{BCSection}
Our interest in the measures $\mu$ measures stems from a link with the study of the dimension and possible absolute continuity of Bernoulli convolutions $\nu_{\beta}$, defined below. We describe here connections with dimension theory for Pisot numbers, links between our work and the question of absolute continuity of $\nu_{\beta}$ for non-Pisot hyperbolic $\beta$ are postponed to a follow up article, in which we generalise \cite{CountingBeta} to give a condition for the absolute continuity of $\nu_{\beta}$ in terms of the growth of $\mu_n([\frac{-1}{\beta-1},\frac{1}{\beta-1}])$, which in turn can be stated in terms of rapid equidistribution to Lebesgue measure of the measures $\mu_n|_{[\frac{-1}{\beta-1},\frac{1}{\beta-1}]}$. We then use the local structure of the measures $\mu_n$ described in Theorem \ref{Thm3} and an analogue of Theorem \ref{GMCocycle} to study this equidistribution.

Given a number $\beta\in(1,2)$, the Bernoulli convolution $\nu_{\beta}$ is the weak$^*$ limit of the measures $\nu_{\beta,n}$ given by
\[
\nu_{\beta,n}=\sum_{a_1\cdots a_n\in\{0,1\}^n} \frac{1}{2^n}\delta_{\sum_{i=1}^na_i\beta^{-i}}
\]
where $\delta_x$ denotes the Dirac probability measure on $x$. The measure $\nu_{\beta}$ is a probability measure on $[0,\frac{1}{\beta-1}]$ and is perhaps the simplest example of a self-similar measure with overlaps. The question of whether $\nu_{\beta}$ is absolutely continuous for some given parameter $\beta$ goes back to Jessen and Wintner \cite{JessenWintner}. Erd\H{o}s showed that $\nu_{\beta}$ is singular when $\beta$ is a Pisot number \cite{ErdosPisot}, and indeed Garsia showed that such Bernoulli convolutions have dimension less than one \cite{GarsiaEntropy}. There has been very substantial progress on the dimension theory of Bernoulli convolutions in the last decade, stemming from the work of Hochman \cite{HochmanInverse}, and in particular it is now known that non-algebraic $\beta$ give rise to Bernoulli convolutions of dimension one \cite{VarjuTranscendental}, whereas for algebraic $\beta$ there are algorithms to determine whether or not $\nu_{\beta}$ has dimension one \cite{VarjuBreuillard1, AFKP}. For a summary of recent research into the dimension theory of Bernoulli Convolutions see \cite{VarjuSummary}. 

There have been many numerical studies into the dimensions of Bernoulli Convolutions associated with Pisot numbers. The evidence we have suggests that for Pisot numbers of large degree the dimension of the corresponding Bernoulli convolution is close to one \cite{AFKP, HKPS, HS1, HS2, KPV}. We formalise this conjecture here.
\begin{conjecture}\label{PisotDimConjecture}
Let $\beta_n$ be a sequence of Pisot numbers in the interval $(1,2)$ and suppose that the degree of $\beta_n$ tends to infinity as $n\to\infty$. Then \[\dim_H(\nu_{\beta_n})\to 1.\]
\end{conjecture}
We have not seen this conjecture formally stated before, but it seems consistent with the (admittedly fairly limited) numerical evidence that we have.

The rest of this section is devoted to giving another conjecture on the measures $\mu_n$ and showing that this new conjecture would be sufficient to prove Conjecture \ref{PisotDimConjecture}. 

It was proved in Hochman \cite{HochmanInverse} that, for algebraic $\beta$ the dimension of the Bernoulli convolution $\nu_{\beta}$ is given by
\[
\dim_H(\nu_{\beta})=\min\left\{1, \frac{H(\beta)}{\log(\beta)}\right\}.
\] 
Here the Garsia entropy $H(\beta)$ is given by
\[
H(\beta):=\lim_{n\to\infty}\frac{1}{n}H_n(\beta)
\]
where
\[
H_n(\beta)=-\sum_{a_1\cdots a_n\in\{0,1\}^n}\frac{1}{2^n}\log\left(\frac{1}{2^n}\#\{b_1\cdots b_n\in\{0,1\}^n: \sum_{i=1}^n (a_i-b_i)\beta^{n-i}=0\}\right).
\]

As noted in \cite{AFKP}, one can use Jensen's inequality to reverse the order of the summation and the log, to get
\begin{eqnarray*}
H_n(\beta)&\geq& -\log\left(\frac{1}{4^n}\#\{a_1\cdots a_n, b_1\cdots b_n\in\{0,1\}^n: \sum_{i=1}^n (a_i-b_i)\beta^{n-i}=0\}\right)\\
&=&\log(4^n)-\log(\mathcal N_n(0)).
\end{eqnarray*}
In particular, our main theorem, Theorem \ref{Thm1}, introduces a constant $\lambda$ equal to the exponential growth rate of $\mathcal N_n(0)$, using this constant we get
\begin{equation}\label{LambdaLowerBound}
H(\beta)\geq \log(4) - \log \lambda.
\end{equation}

Our contribution here in the Pisot case is to link the question of how close to being equidistributed $\mu$ is to the value of $\lambda$, broadly when $\mu|_{[\frac{-1}{\beta-1},\frac{1}{\beta-1}]}$ is well distributed with respect to Lebesgue measure then Equation \ref{LambdaLowerBound} gives a lower bound for the dimension of $\nu_{\beta}$ which is close to one. Our approach here is more or less that of trying to understand something about the maximal eigenvalue of a matrix by studying the corresponding eigenvector. We use the following elementary lemma from linear algebra.
\begin{lemma}\label{LALemma}
Let $M$ be a $k\times k$ matrix with maximal eigenvalue $\rho$ and associated left eigenvector $V=(v_1,\cdots,v_k)$ normalised so that $\sum_{i=1}^k v_i=1$. Let $r_i:=\sum_{j=1}^k M_{i,j}$ denote the $i$th row sum of $M$. Then 
\[
\rho=\sum_{i=1}^k v_ir_i.
\]
\end{lemma}

Let $\beta$ be a Pisot number and $I_{\beta}:=[\frac{-1}{\beta-1},\frac{1}{\beta-1}]$. Then, as noted before, $\lambda$ counts the (weighted) growth of the number of words in $\{-1,0,1\}^n$ for which $\sum_{i=1}^n c_i\beta^{n-i}=0$, the weighting comes from giving each word weight $2^{m}$ where $m$ is the number of occurences of letter $0$ in the word. Whenever $\sum_{i=1}^n c_i\beta^{n-i}=0$ we have that $\sum_{i=1}^m c_i\beta^{m-i}$ is in the interval $I_{\beta}$, and so is in $X(\beta)\cap I_{\beta}$ which is a finite set $V=\{v_1,\cdots,v_k\}$ thanks to the Garsia Separation Property \cite{GarsiaAC}. We write down a matrix $M_0$ indexed by $\{v_1,\cdots ,v_k\}$ with  \[(M_0)_{i,j}=\left\lbrace \begin{array}{cc} 1 & v_j = \beta v_i \pm 1\\
2 & v_j=\beta v_i\\
0 & \mbox{ otherwise} \end{array}\right. .\]

Then the measure $\mu_{I_{\beta}}:=\frac{1}{\mu(I_{\beta})}\mu|_{I_{\beta}}$ gives mass to $v_j$ equal to the $jth$ entry of the left probability eigenvector of $M_0$ associated with maximal eigenvalue $\lambda$. Furthermore, we can read off the $i$th row sum $r_i$ of $M_0$ (associated to point $v_i\in X(\beta)\cap I_{\beta}$) immediately, since we need only know which of $\beta v_i-1, \beta v_i$ and $\beta v_i+1$ lie in $I_{\beta}$. 

Let the function $g_{\beta}:I_{\beta}\to\{1,2,3,4\}$ be given by
\[
g_{\beta}(x)=\chi_{I_{\beta}}(\beta x-1)+2\chi_{I_{\beta}}(\beta x)+\chi_{I_{\beta}}(\beta x+1). 
\]
Then $r_j=g_{\beta}(v_j)$ and so by Lemma \ref{LALemma} we have
\begin{equation}\label{LambdaRowsum}
\lambda=\sum_{v_j\in V}g_{\beta}(v_j)\mu_{I_{\beta}}(v_j)=\int_{I_{\beta}}g_{\beta}(x)d\mu_{I_{\beta}}(x).
\end{equation}
A short calculation gives that if $\mathcal L_{I_{\beta}}$ denotes normalised Lebesgue measure on $I_{\beta}$ then
\[
\int_{I_{\beta}}g_{\beta}(x)d\mathcal L_{I_{\beta}}(x)=\frac{4}{\beta}.
\]
We have the following theorem.
\begin{theorem}\label{EQDimLink}
Let $\beta_n$ be a sequence of Pisot numbers and suppose that
\[
W_1(\mu_{I_{\beta_n}},\mathcal L_{I_{\beta_n}})\to 0
\]
where $W_1$ denotes the Wasserstein metric on the space of probability measures on the Euclidean line. Then $\dim_H(\nu_{\beta_n})\to 1$.
\end{theorem}
\begin{proof}
The function $g_{\beta}$ is a step function on $I_{\beta}$ and it is straightforward to give an upper bound for $|\mu_{I_{\beta}}(A)-\mathcal L_{I_{\beta}}(A)|$ for any of the intervals $A$ upon which the step function is constant in terms of the distance between $\mu_{I_{\beta}}$ and $\mathcal L_{I_{\beta}}$. These upper bounds are uniform in $\beta$. This in turn yields uniform upper bounds on $\int_{I_{\beta}}g_{\beta}d\mu_{I_{\beta}}$, and so by equation \ref{LambdaRowsum} we have a uniform upper bound on $\lambda(\beta)-\log(\frac{4}{\beta})$ in terms of $W_1(\mu_{I_{\beta_n}},\mathcal L_{I_{\beta_n}})$. 

Finally, for Pisot $\beta_n$
\[
\dim_H(\nu_{\beta_n})=\frac{H(\beta_n)}{\log(\beta_n)}\geq \frac{\log 4-\log \lambda(\beta_n)}{\log \beta_n}\to \frac{\log 4 - \log\left(\frac{4}{\beta_n}\right)}{\log(\beta_n)}=1.
\] 
as required.
\end{proof}

The matrix $M_0(\beta)$ associated to a Pisot number $\beta$ is very large for $\beta$ of large degree, and so the numerical evidence we have is limited, but the evidence that we have does suggest that the measures $\mu_{I_{\beta_n}}$ are increasingly well equidistributed for sequences $\beta_n$ of Pisot numbers in $(1,2-\epsilon)$ with degree tending to infinity, see Table \ref{tab:table1}. The $\epsilon$ here is to exclude the multinacci family, which has different behaviour\footnote{Many structures related to the multinacci family $\beta_n^n-\beta_n^{n-1}-\cdots-1=0$, including the spectrum of $\beta_n$, are well understood.}. 
\begin{table}[ht]
  \begin{center}

    \begin{tabular}{r|c|c|c|r} 
      \text{Polynomial} & \text{ $\beta$} & \text{Bound} & \text{$W_1(\mu_\beta,\operatorname{Leb})$}& \text{Matrix Size} \\
     
      \hline
      ${\bf x^3-x^2-x-1}$ & 1.8393 & 0.96422 & 0.13925& 7\\
      $x^3-x^2-1$ & 1.4656 & 0.999116 & 0.0547178& 51\\
      $x^3-x-1$ & 1.3247 & 0.99999 & 0.0286671& 181\\
      ${\bf x^4-x^3-x^2-x-1}$ & 1.9276 & 0.973329 & 0.187067& 9\\
      $x^4-x^3-1$ & 1.3803 & 0.999989 & 0.0149032& 1257\\
      ${\bf x^5-x^4-x^3-x^2-x-1}$ & 1.9659 & 0.983565  & 0.222569& 11\\
      $x^5-x^4-x^3-x^2-1$ & 1.8885 & 0.982269 & 0.0803806& 745 \\
      $x^5-x^4-x^3-x^2+1$ & 1.7785 & 0.995758 & 0.0246573& 951 \\
      $x^5-x^4-x^3-1$ & 1.7049 & 0.993043  &0.0356598& 339 \\
      $x^5-x^4-x^3-x-1$ & 1.8124 & 0.982434 & 0.0571201& 351 \\
      $x^5-x^4-x^3+x^2-1$ & 1.4432 & 0.999982 & 0.00782515& 5423\\
      $x^5-x^4-x^2-1$ & 1.5702 & 0.999862 & 0.0195581& 847\\
      $x^5-x^3-x^2-x-1$ & 1.5342 & 0.999833 & 0.00890312& 2651\\
    \end{tabular}\caption{Pisot numbers $\beta\in(1,2)$ of degree less than six, together with the Wasserstein distance to normalised Lebesgue measure. Multinacci numbers, which have somewhat different behaviour, are in bold.}\label{tab:table1}
  \end{center}
\end{table}


Finally, we give our conjecture on the distribution properties of the measures $\mu_{I_{\beta_n}}$. A proof of this conjecture would imply that Conjecture \ref{PisotDimConjecture} is true by Theorem \ref{EQDimLink}. \begin{conjecture}\label{DimensionConjecture}
Let $\epsilon>0$ and let $(\beta_n)$ be a sequence of Pisot numbers in the interval $(1, 2-\epsilon)$ such that the degree of $\beta_n$ tends to infinity as $n$ tends to infinity. Then the distance 
\[
d(\mu_{I_{\beta_n}},\mathcal L_{I_{\beta_n}})\to 0
\]
as $n\to\infty$, and consequently, by Theorem \ref{EQDimLink}, $\dim_H(\nu_{\beta_n})\to 1$.
\end{conjecture}

\section{A First Example: The Golden Mean}\label{GM}

In this section we prove our main theorems for the special case that $\beta$ is equal to the golden mean $\phi$. Throughout we use the maps $T_i:\mathbb R\to\mathbb R$ given by $T_i(x)=\phi x +i$.

Recall that \[X(\phi)=X_{\{-1,0,1\}}(\phi)=\left\{\sum_{i=1}^n c_i\phi^{n-i}:n\in\mathbb N, c_i\in\{-1,0,1\}\right\}\] and that, for $x\in X(\phi)$,
\[
\mathcal N_n(x):=\#\{a_1\cdots a_n, b_1\cdots b_n \in \{0,1\}^n : \sum_{i=1}^n (a_i-b_i)\phi^{n-i}=x\}
\]
We give the special case of Theorem $\ref{Thm1}$ for when $\beta=\phi$.
\begin{theorem}\label{MeasureExistence}
There exists a number $\lambda>0$ such that limit
\[
\lim_{n\to\infty} \frac{1}{\lambda^n}\mathcal N_n(x)=:\mu(x)
\]
exists for each $x\in X(\phi)$.
\end{theorem}

Here $\lambda$ is easily computed as the maximal eigenvalue of a finite matrix $M_0$ defined below. This theorem will be proved as part of the proof of Theorem \ref{GoldenThm2}.

There are several ways to describe the measure $\mu$. One could construct an infinite transition matrix corresponding to dynamics on $X(\phi)$ induced by the maps $T_0, T_1, T_{-1}$ such that the values of $\mu(x)$ correspond to entries of the eigenvector corresponding to the maximal eigenvalue. In particular, for any finite $K$ we can describe $\mu|_{X(\phi)\cap[-K,K]}$ by reading off the values of an eigenvector of a finite matrix. We give instead a harder construction which allows us to see local structure in the measure $\mu$.

\begin{lemma}\label{MatrixCountingLemma}
There exist matrices $M_0, M_1, M_{-1}$, each of dimensions $17\times 17$ such that for any $x=\sum_{i=1}^n c_i\phi^{n-i}\in X(\phi)$ we have
\[
\mathcal N_n(x)=(M_{c_1}\cdots M_{c_n})_{1,1}
\]
\end{lemma}
\begin{proof}
This proof is similar to the proof of Lemma 3.1 in \cite{AFKP}, we are just using a larger digit set.

If $x=\sum_{i=1}^n c_i\phi^{n-i}$ for some word $c_1\cdots c_n\in\{-1,0,1\}^n$ then we start by tracking words $d_1\cdots d_n\in\{-1,0,1\}^n$ such that
\[
\sum_{i=1}^n c_i\phi^{n-i}=\sum_{i=1} d_i\phi^{n-i},
\]
i.e.
\begin{equation}\label{CDEquality}
\sum_{i=1}^n (c_i-d_i)\phi^{n-i}=0.
\end{equation}
Here $d_i$ represents a difference $a_i-b_i$  where $a_i,b_i\in\{0,1\}$, and so when counting words we want to double count the case $d_i=0$ since it corresponds both to $a_i=b_i=1$ and $a_i=b_i=0$. This accounts for the 2 in the definition of the matrices $M_0, M_1, M_{-1}$.

Now the equality \ref{CDEquality} is equivalent to
\begin{equation}\label{CDEquality2}
T_{c_n-d_n}\circ \cdots \circ T_{c_1-d_1}(0)=0,
\end{equation}
where each $c_i-d_i\in\{-2,-1,0,1,2\}$. The maps $T_i$ are expanding, and in particular if $x\geq 2\phi$ then $T_i(x)\geq 2\phi$, and if $x\leq-2\phi$ then $T_i(x)\leq-2\phi$, for any $i\in\{-2,-1,0,1,2\}$. Thus if equation \ref{CDEquality} holds then for each $m\leq n$ we have
\[
T_{c_m-d_m}\circ \cdots \circ T_{c_1-d_1}(0)\in(-2\phi,2\phi).
\]
By the Garsia separation lemma, or by direct calculation, one can show that there are a finite number of points of the form $T_{c_m-d_m}\circ \cdots \circ T_{c_1-d_1}(0)$ which lie in $(-2\phi,2\phi)$ when $c_i, d_i\in\{-1,0,1\}$. In fact there are 17 such points, we call the set of such possible values $V=\{v_1,\cdots, v_{17}\}$ with $v_1=0$. 

Now in general the difference $c_i-d_i$ can take values in $\{-2,-1,0,1,2\}$, but if we know the value of $c_i$ then $c_i-d_i$ can only take three of these values, if $c_i=1$ then $c_i-d_i$ can take values $0$ $1$ or $2$ for example. 

Let $M_1$ be the $17\times 17$ matrix with rows and columns indexed by elements of $V$, with
\[
(M_1)_{ij}=\left\lbrace\begin{array}{cc}1 & v_j=T_0 (v_i) \text{ or } v_j=T_{-2} (v_i)\\
2 & v_j=T_{-1} (v_i)\\
0 & \text{ otherwise } \end{array}\right. 
\]
This is the transition matrix for the maps $T_{c_i-d_i}$ where we know $c_i=1$ and $d_i\in\{-1,0,1\}$, the values $1$ and $2$ occur because we have one way of letting $d_i=a_i-b_i$ equal $1 $ or $-1$ but two ways of letting $d_i=0$.

Similarly, let $M_{-1}$ be the matrix with rows and columns indexed by elements of $V$, with
\[
(M_{-1})_{ij}=\left\lbrace\begin{array}{cc}1 & v_i=T_0 (v_j) \text{ or } v_i=T_{2} (v_j)\\
2 & v_i=T_{1} (v_j)\\
0 & \text{ otherwise } \end{array}\right.
\]
and let $M_{0}$ be the matrix with rows and columns indexed by elements of $V$, with
\[
(M_{0})_{ij}=\left\lbrace\begin{array}{cc}1 & v_i=T_1 (v_j) \text{ or } v_i=T_{-1} (v_j)\\
2 & v_i=T_{0} (v_j)\\
0 & \text{ otherwise } \end{array}\right. .
\]
Then given $c_1,\cdots c_n\in\{-1,0,1\}^n$, the $(i,j)$th term of the matrix $M_{c_n}\cdots M_{c_1}$ represents the number of $d_1\cdots d_n\in\{-1,0,1\}$ for which
\begin{equation}\label{CD3}
T_{c_n-d_n}\circ\cdots T_{c_1-d_1}(v_i)=v_j. 
\end{equation}

Again here when we refer to the `number' of $d_1\cdots d_n$ we are double counting when $d_i=0$ because we have two ways of putting $a_i-b_i=0$. 

Thus in order to count equalities of the form (\ref{CDEquality2}), we need to use (\ref{CD3}) with $v_i=v_j=v_{1}=0. $ We conclude that the number of $a_1\cdots a_n, b_1\cdots b_n$ such that $\sum_{i=1}^n (a_i-b_i)\phi^{n-i}=x$ is given by the top left entry of the matrix $M_{c_n}\cdots M_{c_1}$, where $c_1\cdots c_n$ is any $\{-1,0,1\}$ code for which $x=\sum_{i=1}^n c_i\phi^{n-i}$.

\end{proof}
We now state and prove Theorem \ref{Thm2} for the special case that $\beta$ is equal to $\phi$.
\begin{theorem}\label{GoldenThm2}
Let $W$ be the left eigenvector of $M_0$ corresponding to the maximal eigenvalue $\lambda$. Then for any $x=\sum_{i=1}^n c_i\phi^{n-i}\in X(\phi)$ we have
\[
\mu(x)=\frac{1}{\lambda^n}(WM_{c_1}M_{c_2}\cdots M_{c_n})_1,
\]
that is, $\lambda^n\mu(x)$ is the first entry in the $1\times 17$ vector $WM_{c_1}\cdots M_{c_n}$. 
\end{theorem}

\begin{proof}
In the previous lemma we showed how to count the number of words $a_1,\cdots a_n, b_1\cdots b_n$ with $\sum_{i=1}^n (a_i-b_i)\phi^i=x$, given knowledge of one code $c_1\cdots c_n\in\{-1,0,1\}^n$ such that \begin{equation}\label{cEq}x=\sum_{i=1}^n c_i\phi^{n-i}.\end{equation}

Here it was important that the length of the word $c_1\cdots c_n$ coding $x$ corresponded with the $\mathcal N_n$ which we want to calculate. But if equation $\ref{cEq}$ holds then it is also true that
\[
x=\sum_{i=1}^n c_i\phi^{n-i} + 0\phi^n + 0\phi^{n+1}+\cdots + 0\phi^{n+(k-1)}.
\]
So again using Lemma \ref{MatrixCountingLemma} we see that
\begin{eqnarray*}
\mathcal N_{n+k}(x)&=&(M_0^kM_{c_1}\cdots M_{c_n})_{1,1}\\
&= & (1~0~0\cdots) M_0^k M_{c_1}\cdots M_{c_n} \left(\begin{array}{c}1\\0\\0\\\vdots\end{array}\right). 
\end{eqnarray*}
If $\lambda$ is the maximal eigenvalue of $M_0$ then, since $M_0$ is primitive, there exists a corresponding eigenvector $W$ such that
\[
\frac{1}{\lambda^k}(1~0~0\cdots) M_0^k\to W
\]
Putting the previous equations together gives that if $x=\sum_{i=1}^n c_i\phi^{n-i}$ then
\begin{eqnarray*}
\mu(x)&=&\lim_{k\to\infty} \frac{1}{\lambda^{n+k}}\mathcal N_{n+k}(x)\\
&=&\lim_{k\to\infty}\frac{1}{\lambda^{k}}\frac{1}{\lambda^n}(1~0~0\cdots) M_0^k M_{c_1}\cdots M_{c_n} \left(\begin{array}{c}1\\0\\0\\\vdots\end{array}\right)\\
&=& \frac{1}{\lambda^n}WM_{c_1}\cdots M_{c_n} \left(\begin{array}{c}1\\0\\0\\\vdots\end{array}\right).
\end{eqnarray*}
\end{proof}
It is also important to note that if $x=\sum_{i=1}^n c_i\phi^{n-i}$ then the vector $\frac{1}{\lambda^n}WM_{c_1}\cdots M_{c_n}$ doesn't just hold information on $\mu(x)$, which is the first entry, but also holds information on the values of $\mu$ at other elements of $X(\phi)$. 
\begin{lemma}\label{vk}
For $v_k$ the $kth$ element of $V$ we have
\[
\mu(x+v_k)=\frac{1}{\lambda^n}(WM_{c_1}M_{c_2}\cdots M_{c_n})_k,
\]
that is, $\lambda^n\mu(x+v_k)$ is the $k$th entry in the $1\times 17$ vector $WM_{c_1}\cdots M_{c_n}$. 
\end{lemma} 
\begin{proof}
This follows directly from the proof of the previous lemma and equation \ref{CD3}.
\end{proof}
This allows us to start to discuss local structure for $\mu$. We want to describe how one can use dynamics to move through the measure $\mu$ and write down the set of pairs $\{(x,\mu(x)):x\in X(\phi)\}$. To do this, we must first recall the cut and project structure of the set $X(\phi)$.

\subsection{The Structure of $X(\phi)$}

The work of this subsection is well known to experts. We first show that set $X(\phi)$ can be dynamically generated. One can move from a level-$n$ sum to a level-$(n+1)$ sum in the construction of $X(\phi)$ by observing that
\[
\sum_{i=1}^{n+1} c_i\phi^{n+1-i}=\phi\left(\sum_{i=1}^n c_i\phi^{n-i}\right)+c_{n+1}.
\]
Thus with $T_i(x):=\phi x+i$ as before we see that
\begin{equation}\label{Tform}
X(\phi)=\{T_{c_n}\circ\cdots \circ T_{c_1}(0):n\in\mathbb N, c_i\in\{-1,0,1\}\}. 
\end{equation}
As $\phi^2=\phi+1$ we can consider multiplication by $\phi$ in terms of its action on numbers of the form $z_1\phi+z_0$. We let $\pi_e:\mathbb Z^2\to \mathbb R$ be given by \[\pi_e\left(\begin{array}{c}z_1\\z_0\end{array}\right):=z_1\phi+z_0\]
and $\pi_c:\mathbb Z^2\to\mathbb R$ be given by
\[
\pi_c\left(\begin{array}{c}z_1\\z_0\end{array}\right):=\frac{-1}{\phi}z_1+z_0.
\]
We will later refer to $\pi_e$ as projection in the expanding direction and $\pi_c$ as projection in the contracting direction. Note that $\pi_e:\mathbb Z^2\to \mathbb R$  and $\pi_c:\mathbb Z^2\to\mathbb R$ are injective (if they were not then $x^2-x-1$ would not be the minimal polynomial of $\phi$).

Then
\[
\phi\left(\pi_e\left(\begin{array}{c}z_1\\z_0\end{array}\right)\right)= z_1\phi^2+z_0\phi= (z_1+z_0)\phi+z_1=\pi_e\left(\left(\begin{array}{cc}1& 1\\1& 0\end{array}\right)\left(\begin{array}{c}z_1 \\ z_0\end{array}\right)\right)
\] 
and so $T_i:X(\phi)\to X(\phi)$ lifts to a map $\tilde T_i:\mathbb Z^2\to\mathbb Z^2$ given by
\[
\tilde T_i\left(\begin{array}{c}z_1\\z_0\end{array}\right)=\left(\begin{array}{cc}1& 1\\1& 0\end{array}\right)\left(\begin{array}{c}z_1 \\ z_0\end{array}\right)+\left(\begin{array}{c}0\\i\end{array}\right).
\]
We let 
\[
\tilde X(\phi):=\left\lbrace\tilde T_{c_n}\circ\cdots\circ \tilde T_{c_1}\left(\begin{array}{c}0\\0\end{array}\right):n\in\mathbb N, c_i\in\{-1,0,1\}\right\rbrace
\]
and have the relation $X(\phi)=\pi_e(\tilde X(\phi))$.

One can study the structure of $X(\phi)$ directly on the real line, this was done for example in \cite{FengWen} where the substitution structure of $X(\phi)$ was described. However, some properties of $X(\phi)$ are easier to see if we first study the structure of $\tilde X(\phi)$. For example, from equation (\ref{Tform}) we see that the uniformly discrete set $X(\phi)$ is a subset of the dense set $\{z_1\phi+z_0:z_1,z_0\in\mathbb Z\}$, but it is not immediately apparent which values of $(z_1,z_0)$ correspond to points in $X(\phi)$.

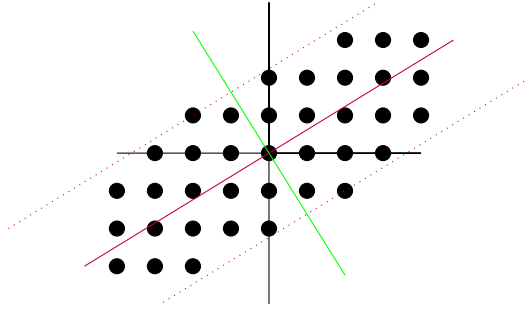
\begin{figure}\label{GoldenLift}
\centering
\begin{tikzpicture}[scale=0.5]
\draw(-4,0)--(0,0)--(0,4)--(0,0)--(4,0)--(0,0)--(0,-4);
\node[draw,circle,inner sep=2pt,fill] at (0,0){};
\node[draw,circle,inner sep=2pt,fill] at (0,1){};
\node[draw,circle,inner sep=2pt,fill] at (0,-1){};
\node[draw,circle,inner sep=2pt,fill] at (-1,-1){};
\node[draw,circle,inner sep=2pt,fill] at (-1,0){};
\node[draw,circle,inner sep=2pt,fill] at (-1,1){};
\node[draw,circle,inner sep=2pt,fill] at (1,-1){};
\node[draw,circle,inner sep=2pt,fill] at (1,0){};
\node[draw,circle,inner sep=2pt,fill] at (1,1){};
\node[draw,circle,inner sep=2pt,fill] at (1,2){};
\node[draw,circle,inner sep=2pt,fill] at (2,1){};
\node[draw,circle,inner sep=2pt,fill] at (2,0){};
\node[draw,circle,inner sep=2pt,fill] at (2,2){};
\node[draw,circle,inner sep=2pt,fill] at (0,2){};
\node[draw,circle,inner sep=2pt,fill] at (-1,-2){};
\node[draw,circle,inner sep=2pt,fill] at (-2,-1){};
\node[draw,circle,inner sep=2pt,fill] at (-2,0){};
\node[draw,circle,inner sep=2pt,fill] at (-2,-2){};
\node[draw,circle,inner sep=2pt,fill] at (0,-2){};
\node[draw,circle,inner sep=2pt,fill] at (3,1){};
\node[draw,circle,inner sep=2pt,fill] at (3,2){};
\node[draw,circle,inner sep=2pt,fill] at (3,0){};
\node[draw,circle,inner sep=2pt,fill] at (3,3){};
\node[draw,circle,inner sep=2pt,fill] at (2,3){};
\node[draw,circle,inner sep=2pt,fill] at (4,2){};
\node[draw,circle,inner sep=2pt,fill] at (4,3){};
\node[draw,circle,inner sep=2pt,fill] at (4,1){};
\node[draw,circle,inner sep=2pt,fill] at (2,-1){};
\node[draw,circle,inner sep=2pt,fill] at (-3,-1){};
\node[draw,circle,inner sep=2pt,fill] at (-3,-2){};
\node[draw,circle,inner sep=2pt,fill] at (-3,0){};
\node[draw,circle,inner sep=2pt,fill] at (-3,-3){};
\node[draw,circle,inner sep=2pt,fill] at (-2,-3){};
\node[draw,circle,inner sep=2pt,fill] at (-4,-2){};
\node[draw,circle,inner sep=2pt,fill] at (-4,-3){};
\node[draw,circle,inner sep=2pt,fill] at (-4,-1){};
\node[draw,circle,inner sep=2pt,fill] at (-2,1){};
\draw[green](-2, 3.236)--(2, -3.236);
\draw[purple](4.854,3)--(-4.854,-3);
\draw[dotted,purple] (-6.854,-2)--(2.854,4);
\draw[dotted,purple] (6.854,2)--(-2.854,-4);
\end{tikzpicture}
\caption{The set $\tilde X(\phi)$ around the origin, with expanding and contracting eigenvectors shown}
\end{figure}

Lifting to $\tilde X(\phi)$ the structure becomes clear. The matrix $\left(\begin{array}{cc}1& 1\\1& 0\end{array}\right)$ has one expanding eigenvector and one contracting eigenvector, and the maps $\tilde T_i$ can be described in terms of their action on points written in terms of these eigenvectors.  


Note that if $\pi_c\left(\begin{array}{c}z_1\\z_0\end{array}\right)=x$ then \[\pi_c(\tilde T_i\left(\begin{array}{c}z_1\\z_0\end{array}\right))=\frac{-x}{\phi}+i=:S_i(x).\] Then the system $\{S_0, S_1, S_{-1}\}$ is a contracting iterated function system with attractor $[-\phi^2,\phi^2]$, and so for any point $\left(\begin{array}{c}z_1\\z_0\end{array}\right)=\tilde T_{a_n}\circ \cdots \tilde T_{a_1}\left(\begin{array}{c}0\\0\end{array}\right)\in X(\phi)$ we have $\pi_c\left(\begin{array}{c}z_1\\z_0\end{array}\right)= S_{a_n}\circ \cdots S_{a_1}(0)\in (-\phi^2,\phi^2)$. The converse is also true and is contained in the following lemma.

\begin{lemma}\label{XPhiStructure}
The set $\tilde X(\phi)$ consists of all pairs $\left(\begin{array}{c}z_1\\z_0\end{array}\right)\in\mathbb Z^2$ for which $\pi_c\left(\begin{array}{c}z_1\\z_0\end{array}\right)$  lies in the interval $(-\phi^2,\phi^2)$. 

Furthermore, if $\pi_c\left(\begin{array}{c}z_1\\z_0\end{array}\right)\in S_{d_1}\circ \cdots S_{d_k}(-\phi^2,\phi^2)$ for some $d_1,\cdots, d_k\in\{-1,0,1\}^k$ then for all sufficiently large $n$ there exists a word $c_1\cdots c_{n+k}\in\{-1,0,1\}^{n+k}$ with $c_{n+k}\cdots c_1=d_1\cdots d_k$ and such that 
\[
\left(\begin{array}{c}z_1\\z_0\end{array}\right)=\tilde T_{c_{n+k}}\circ \cdots \circ T_{c_1}\left(\begin{array}{c}0\\0\end{array}\right)
\]
\end{lemma}


\begin{proof}
One inclusion was proved in the paragraph before the statement of this lemma.

Now let $(z_1,z_0)\in\mathbb Z^2$ have $\pi_c(z_1,z_0)\in (-\phi^2,\phi^2)$. We wish to find a word $c_1\cdots c_n$ such that \[\left(\begin{array}{c}z_1\\z_0\end{array}\right)=\tilde T_{c_n}\circ \cdots \tilde T_{c_1}\left(\begin{array}{c}0\\0\end{array}\right),\] or equivalently
\begin{equation}\label{return}
\left(\begin{array}{c}0\\0\end{array}\right)=\tilde T_{c_1}^{-1}\circ\cdots \tilde T_{c_n}^{-1}\left(\begin{array}{c}z_1\\z_0\end{array}\right).
\end{equation}

We first observe that for any $\left(\begin{array}{c}z_1\\z_0\end{array}\right)$ with $\pi_c\left(\begin{array}{c}z_1\\z_0\end{array}\right)\in(-\phi^2,\phi^2)$ and $\pi_e\left(\begin{array}{c}z_1\\z_0\end{array}\right)\in[-\phi,\phi]$ one can find words $c_1\cdots c_n$ such that Equation \ref{return} holds. Since there are only finitely many pairs $\left(\begin{array}{c}z_1\\z_0\end{array}\right)$ in this bounded region one can check this observation with a finite calculation.

Now let $\left(\begin{array}{c}z_1\\z_0\end{array}\right)$ have $\pi_c\left(\begin{array}{c}z_1\\z_0\end{array}\right)\in(-\phi^2,\phi^2)$, but place no restriction on $\pi_e\left(\begin{array}{c}z_1\\z_0\end{array}\right)\in(-\phi,\phi)$. By the IFS construction of the contracting interval, we can choose arbitrarily long words  $i_1\cdots i_n\in\{-1,0,1\}$ such that $\tilde T_{i_n}^{-1}\circ\cdots \tilde T_{i_1}^{-1}(\left(\begin{array}{c}z_1\\z_0\end{array}\right))$ still has contracting coordinate in the interval $(-\phi^2,\phi^2)$. But since inverse maps $\tilde T_i^{-1}$ contract the expanding direction, the expanding coordinate will eventually lie in $[-\phi,\phi]$, and by the previous paragraph we know that we can return to $\left(\begin{array}{c}0\\0\end{array}\right)$. Finally we not that if we had $\pi_c\left(\begin{array}{c}z_1\\z_0\end{array}\right)\in S_{d_1}\circ \cdots S_{d_k}(-\phi^2,\phi^2)$ then we can choose the word $i_1\cdots i_n$ to start with $d_1\cdots d_k$.

\end{proof}

It is worth stressing that the first three quarters of the preceeding proof generalises easily to any algebraic integer $\beta$, but the finite check that any integer pair suitably close\footnote{For hyperbolic non-Pisot $\beta$ we will also require that expansions of Galois conjugates are close to the origin, see section \ref{GeneralSec}.} to the origin can return to the origin under the maps $\tilde T_i^{-1}$ needs verifying for each $\beta$ and we don't know that it is always true. 

One interesting consequence of Lemma \ref{XPhiStructure} is that in order to understand the distance from some point $\tilde T_{c_n}\circ \cdots \tilde T_{c_1}\left(\begin{array}{c}0\\0\end{array}\right)$ to its close neighbours in $\tilde X(\phi)$, we need only to know about $\pi_c( \tilde T_{c_n}\circ \cdots \tilde T_{c_1}\left(\begin{array}{c}0\\0\end{array}\right))$. 

Given $x\in X(\phi)$ let $\tilde x$ denote the corresponding point in $\tilde X(\phi)$ and let $x_c=\pi_c(\tilde x)$. For $K\in\mathbb R$ let $x\in X(\phi)$. Call the set 
\[
(X(\phi)-x)\cap[-K,K]=\{y-x: y\in X(\phi), y-x\in[-K,K]\}
\]
the $K$-neighbourhood of $x$. 

\begin{lemma}\label{XPhiLocalStructure}[Local Structure for $X(\phi)$]
For any $K>0$ there exists a finite partition of $(-\phi^2,\phi^2)$ such that the $K$-neighbourhood of any $x\in X(\phi)$ depends only upon which partition element of $(-\phi^2,\phi^2)$ $x_c$ lies in.
\end{lemma}
\begin{proof}
This follows from the analagous statement for $\tilde X(\phi)$, which has a fairly direct proof following Lemma \ref{XPhiStructure}, since one needs only to consider which translations in $\mathbb Z^2$ can be performed without leaving the contracting window or moving by a distance of more than $K$ in the expanding direction. 
\end{proof}


Finally, we outline how to use dynamics to describe the odometer map which maps $x\in X(\phi)$ to $\min\{y\in X(\phi):y>x\}$.

Let $d:X(\phi)\to \mathbb R^+$ denote the distance from $x\in X(\phi)$ to $\min\{y\in X(\phi):y>x\}$. That is, let $d$ be defined by
\[
d(x)=\min\{y\in X(\phi):y>x\}-x.
\]

\begin{prop}\label{XEvolution}
The odometer map $x\to x+d(x)$ on $X(\phi)$ lifts to the skew-product map $O:X(\phi)\times X_c(\phi)\to X(\phi)\times X_c(\phi)$ by
\[
\tilde d(x,x_c)=\left \lbrace \begin{array}{cc} (x+2\phi-3, x_c-\frac{2}{\phi}-3) & x_c\in[\phi,\phi^2 ]\\
(x+\phi-1, x_c-1-\frac{1}{\phi}) & x_c \in (0,\phi)\\
(x+2-\phi, x_c+2+\frac{1}{\phi}) & x_c\in [-\phi^2,0 ] \end{array} \right.
\]
\end{prop}
We stress here that the action of $O$ on the contracting direction is of a uniquely ergodic interval exchange transformation.
\begin{proof}
The fact that there is some partition of $(-\phi^2, \phi^2)$ telling us how to evolve a skew-product map which is a lift of $d$ follows immediately from Lemma \ref{XPhiLocalStructure} with $K=\phi-1$. It is a finite calculation to write down the map exactly.
\end{proof}

\subsection{An Odometer map for $\mu$}
Proposition \ref{XEvolution} dealt with how one can move locally through the set $X(\phi)$ using only knowledge on the position in the contracting direction, we want to build a similar theorem which also incorporates knowlede of the values $\mu(x)$, we do this by building a cocycle over the odometer map $O$.

Given $x\in X(\phi)$ let $x_c$ denote the corresponding point in the contracting window $(-\phi^2,\phi^2)$. We recall from Lemma \ref{XPhiStructure} that for $x\in X(\phi)$ and for any word $d_{1}\cdots d_{k}$, $x$ can be written $x=\sum_{i=1}^n c_i\phi^{n-i}$ with $c_{n-k+1}\cdots c_n=d_k\cdots d_1$ if and only if $x_c\in S_{d_1}\circ \cdots \circ S_{d_n}(-\phi^2,\phi^2)$. 

Now let us map real $1\times 17$ vectors $U$ with strictly positive first entry onto the corresponding projective space by letting
\[
(U')_i=\dfrac{(U)_{i+1}}{(U)_1}
\]
for $(1\leq i\leq 16)$. in particular, we associate to each $x=\sum_{i=1}^n c_i\phi^{n-i}\in X(\phi)$ the corresponding vector $V(x)=(WM_{c_1}M_{c_2}\cdots M_{c_n})'$ considered as an element of real projective space. To be concrete, we define the $1\times 16$ vector $V(x)$ by
\[
(V(x))_i=\dfrac{(WM_{c_1}M_{c_2}\cdots M_{c_n})_{i+1}}{(WM_{c_1}M_{c_2}\cdots M_{c_n})_1}=\dfrac{\mu(x+v_i)}{\mu(x)}.
\]
It follows from the proofs of the previous two statements that these vectors do not depend on the choice of code $c_1\cdots c_n$ of $x$. We can also write $V(x)$ as a function $V(x_c)$ of the position in the contracting window. 

Consider the metric $d$ on the space of $1\times 16$ non-negative vectors by letting
\[
d(U,V)=\max_{i\in\{1,\cdots 16\}}|\ln(v_i)-\ln (u_i)|.
\]
Two vectors $U,V$ are at infinite distance from one another if there exist $i,j\in\{1\cdots 16\}$  such that $u_i=0$ $v_i\neq 0$ or $v_i=0$ $u_i\neq 0$.
\begin{lemma}\label{MatrixCondition}
Suppose that $A$ is a $17\times 17$ matrix with $A_{1,1}>0$ such that for any pair of parameters $(i,j)\in\{1,\cdots,17\}^2$ one of the following holds
\begin{enumerate}
\item $(i,j)$ is in a zero row, i.e. $(A)_{i',j}=0$ for all $i'\in\{1,\cdots,17\}$
\item $(i,j)$ is in a zero column, i.e. $(A)_{i,j'}=0$ for all $j'\in\{1,\cdots,17\}$
\item $(A)_{i,j}>0$.
\end{enumerate}
Then there exists a constant $C<1$ such that, for any $1\times 17$ vectors $U,V$ with positive first entries and with $d(U',V')<\infty$ we have
\[
d((UA)',(VA)')<Cd(U',V').
\]
Furthermore, there exists $K>0$ such that, for any any $1\times 17$ vectors $U,V$ with positive first entry (and possibly with $d(U',V')=\infty$), 
\[
d((UA)',(VA)')<K.
\] 
\end{lemma}
This lemma is proved carefully in section \ref{GeneralSec}. 
\begin{lemma}\label{M_0}
The matrix $M_0^7$ satisfies the condition of Lemma \ref{MatrixCondition}.
\end{lemma}
This can be verified by a short calculation.

One can also see that given a $17\times 17$ non-negative matrix $B$ withstrictly positive top left entry and two $1\times 17$ vectors $U$ and $V$ with strictly positive first entries,
\[
d((UA)',(VA)')\leq d(U',V').
\]
This shows that matrices $M_0$, $M_1$ and $M_{-1}$ do not expand distances between vectors in our metric.

Finally we are able to state Theorem \ref{Thm3} in the special case that $\beta=\phi$ and dealing only with nearest neighbours. Recall that, for $x\in X(\phi)$, $d(x):=\min\{y-x: y\in X(\phi),y>x\}$.

\begin{prop}
For $x\in X(\phi)$ with corresponding point $x_c\in(-\phi^2,\phi^2)$ define $f(x_c)$ by
\[
\ln(\mu(x+d(x)))-\ln (\mu(x))=f(x_c).
\]
Then $f$ is bounded and is continuous at each $x_c\in X_c(\phi)$ except for $0$ and $\phi$.  
\end{prop}
If we defined $d'$ on $(\phi^2,\phi^2)$ by $d'(x_c):=d(x)$ then $0$ and $\phi$ are the points in $(\phi^2,\phi^2)$ where $d'(x_c)$ is not continuous.
\begin{proof}
We have already shown that 
\[
d(x)=\left\lbrace\begin{array}{cc} 2\phi-3 & x_c\in[\phi,\phi^2)\\
\phi-1 & x_c\in(0,\phi)\\
2-\phi & x_c\in(-\phi^2,0] \end{array}\right.
\]
One can check that each of $2\phi-3, \phi-1$ and $2-\phi$ correspond to entries $v_k$ of $V$. Then by Lemma \ref{vk} we see that
\[
f(x_c):=\ln(\mu(x+d(x)))-\ln(\mu(x))
\]
appears as the log of a ratio of two entries in the vector $(WM_{c_1}\cdots M_{c_n})$ for any $c_1\cdots c_n$ coding $x$. Since both $x$ and $x+d(x)$ have strictly positive mass, the difference of the logs is finite so $f(x_c)\in\mathbb R$.

We now discuss the continuity properties of $f$. Let $x\in X(\phi)$ and $\epsilon>0$ be given. Let $K$ and $C$ be the quantities introduced in Lemma \ref{MatrixCondition} associated to $M_0^7$, and let $r\in\mathbb N$ be such that $KC^{r-1}<\epsilon$.  Let $c_1\cdots c_n$ be a code of $x$ containing at least $r$ copies of the word $0000000$, this can be done for example by taking any expansion of $x$ and adding lots of zeros to the start. 

Now $x_c$ is contained in the interval $S_{c_n}\circ S_{c_{n-1}}\circ \cdots \circ S_{c_1}(-\phi^2,\phi^2)$. Let $y\in X(\phi)$ be another point with $y_c\in S_{c_n}\circ S_{c_{n-1}}\circ \cdots \circ S_{c_1}(-\phi^2,\phi^2)$. Then $y$ can be written $y=\sum_{d=1}^m d_i\phi^{m-i}$ for some code $d_1\cdots d_m$ with $d_{m-n}\cdots d_n=c_1\cdots c_n$, as in Lemma \ref{XPhiStructure}.  

Assume that $x_c$ and $y_c$ lie in the same one of the intervals $(-\phi^2,0],$$(0,\phi)$, $[\phi,\phi^2)$ so that $d(x)=d'(x_c)=v_j$. Then
\begin{eqnarray*}
|f(x_c)-f(y_c)|&=& |\ln(WM_{c_1}\cdots M_{c_n})_j-\ln(WM_{d_1}\cdots M_{d_m})_j|\\
&=& |\ln(WM_{c_1}\cdots M_{c_n})_j-\ln(WM_{d_1}\cdots M_{d_{m-n-1}}M_{c_1}\cdots M_{c_n})_j|\\
&\leq & d((WM_{c_1}\cdots M_{c_n})', (\underbrace{WM_{d_1}\cdots M_{d_{m-n-1}}}_{=:U}M_{c_1}\cdots M_{c_n})')\\
&=& d((WM_{c_1}\cdots M_{c_n})', (UM_{c_1}\cdots M_{c_n})')
\leq KC^{r-1}<\epsilon. 
\end{eqnarray*}
Here the final line follows since $c_1\cdots c_n$ contains $r$ non-overlapping occurences of the word $M_0^7$, the first of which guarantees that \[d((WM_{c_1}\cdots M_{c_n})', (UM_{c_1}\cdots M_{c_n})')<K\] and the subsequent $r-1$ of which multiply this upper bound by $C$, thanks to Lemmas \ref{MatrixCondition} and \ref{M_0}. \end{proof}

We have now completed the proofs of analogues of Theorems \ref{Thm1}, \ref{Thm2}, and \ref{Thm3} in the special case of the golden mean, although the analogue of \ref{Thm3} we did only for moving to nearest neighbours. 

Putting everything together, we get the following theorem which demonstrates how one can move through the measure $\mu$ on $X(\phi)$, and how one could start to study it using ergodic theory.

\begin{theorem}\label{GMCocycle}
Let the map $\psi: X(\phi)\times(-\phi^2,\phi^2)\times \mathbb R$ be given by
\[
\phi (x,y,z)=\left\lbrace\begin{array}{cc} (x+2\phi-3,y-\frac{2}{\phi}-3,z+f(y)) & y\in[\phi,\phi^2)\\
(x+\phi-1, y-\frac{1}{\phi}-1,z+f(y)) & y\in(0,\phi)\\
(x+2-\phi, y+2+\frac{1}{\phi}, z+f(y)) & y\in(-\phi^2,0] \end{array}\right.
\]
Then if $x$ is the $n$th element to the right of $0$ in $X(\phi)$ we have that
\[
(x,x_c,\mu(x))=\psi^n(0,0,0).
\]
\end{theorem}
Thus we have that many of the properties of $\mu$ can be studied by studying $\psi$, which is really a skew-product over an interval exchange transformation on the contracting window $(\-\phi^2,\phi^2)$.

\section{Measures on the spectra of general hyperbolic algebraic integers}\label{GeneralSec}

In this section we show how to extend the previous work to general hyperbolic algebraic integers and prove Theorems \ref{Thm1}, \ref{Thm2} and \ref{Thm3}. As stated in the introduction, the motivation is to study measures of the form

\[
\mu_n(x)=\dfrac{1}{4^n}\#\{a_1\cdots a_n, b_1\cdots b_n \in \{0,1\}^n: \sum_{i=1}^n (a_i-b_i)\beta^{n-i}=x\}.
\]

Given $\beta$, we lift $\mu_n$ to a measure $\bar{\mu}_n$ living on a lattice subset of a multidimensional euclidean space $\mathbb{K}$. We prove that there is $\gl>0$ such that $4^n\bar{\mu}_n/\gl^n$ converges to a measure $\bar{\mu}$. We also prove that there are local patterns in the measure $\bar{\mu}$ that repeat in a way that we understand. This means that we understand how the measure of a lattice point changes when we move to nearby points on the lattice\footnote{We don't state an analogue of Theorem \ref{GMCocycle} for the higher dimensional case since there is no natural choice of `next point' to move to when we are working in higher dimensional Euclidean space. One could state such results, perhaps by identifying a strip which is infinite in only one direction and describing the dynamics to move through such a strip.}. In particular there is a non-trivial linear subspace $\mathbb{K}_c$ of $\mathbb{K}$ such that the following holds. Under conditions and given a suitable vector $d$ then for typical $x$ the ratio $\frac{\bar{\mu}(x+d)}{\bar{\mu}(x)}$ is determined, up to certain accuracy, by the approximate position of the orthogonal projection of $x$ on $\mathbb{K}_c$. That is the numbers of the form $\frac{\bar{\mu}(x+d)}{\bar{\mu}(x)}$ are approximately equal for all $x$ projecting on to the same small region of $\mathbb{K}_c$.

Let $\gb=\gb_1\in(1,2)$ be an algebraic integer with Galois conjugates $\gb_2,...,\gb_d,\gb_{d+1},...,\gb_{d+s}$ such that $|\gb_2|,...,|\gb_d|>1$ and $|\gb_{d+1}|,...,|\gb_{d+s}|\in(0,1)$. Further define $\bar{\beta}^n=(\gb_1^n,...,\gb_{d+s}^n)$. For this section we let

\[T_i(x_1,...,x_{d+s})=(\beta_1x_1+i,...,\beta_{d+s}x_{d+s}+i),\] 

these maps are higher dimensional lifts of their analogues in the previous section. For Galois conjugates $\gb_i\in\mathbb{C}$ let $\mathbb{F}_{\gb_i}=\mathbb R$ if $\gb_i\in\mathbb R$ and $\mathbb{F}_{\gb_i}=\mathbb C$ if $\gb_i\in\mathbb C\setminus\mathbb R$. We define the sets 

\begin{align*}
\mathbb{K}&:=\prod_{i=1}^{d+s}\mathbb{F}_{\gb_i},\\
	\mathbb{K}_c&:=\{0\}^{d}\times \mathbb{F}_{\gb_{d+1}}\times...\times \mathbb{F}_{\gb_{d+s}},\\
\bar{Z}&:=\{a_{d+s-1}\bar{\gb}^{d+s-1}+...+a_0\bar{\gb}^0:a_{d+s-1},...,a_0\in\mathbb{Z}\}, 	
\end{align*}

and 
 \begin{align*}
 \bar{X}(\beta)&:=\left\{\sum_{i=1}^{n}a_i\bar{\gb}^{n-i}:n\in\mathbb{N},a_1...,a_n\in\{-1,0,1\}\right\}
 \\&=\left\{T_{a_n}\circ...\circ T_{a_1}( 0):n\in\mathbb{N},a_1...,a_n\in\{-1,0,1\}\right\}
 \end{align*}
 where $ 0$ denotes the origin in $\mathbb K$.
 
The set $\bar{Z}$ is a lattice in $\mathbb{K}\cong\mathbb{R}^{\sum_{i=1}^{d+s}\text{dim}(\mathbb{F}_{\gb_i})}$. That is because $\{\bar{\beta}^0,...,\bar{\beta}^{d+s-1}\}$ is an independent subset of the real vector space $\mathbb{K}$. That can be checked using the formula for the determinant of the Vandermonde matrix. It is useful to keep in mind that for each $i\in\mathbb{Z}$ we have $T_i(\bar{Z})\subseteq\bar{Z}$, in particular $\bar{X}(\beta)\subseteq\bar{Z}$. 


Notice that all coordinate projections, restricted on $\bar{Z}$, are injective so there is in a sense a natural identification of $\bar{Z}$ to any image of it under a coordinate projection. Here by a coordinate projection we mean any map from $\mathbb{K}$ to itself, of the form $(a_1,...,a_{d+s})\mapsto(a_1\gk_1,...,a_{d+s}\gk_{d+s})$ where $\gk_1,...,\gk_{d+s}\in\{0,1\}$. As in the one dimensional case, we define the measure $\bar{\mu}_n$ on $\bar{Z}$ by
 \[
 \bar{\mu}_{n}(x)=\frac{1}{4^n}\bar{\mathcal N}_n(x)
 \]
 where
 \begin{align*}
 \bar{\mathcal N}_{n}(x)=&\#\left\{(a_1,...,a_{n},b_1,...,b_{n})\in\{0,1\}^{2n}:\sum_{i=1}^{n}a_i\bar{\gb}^{n-i}-\sum_{i=1}^{n}b_i\bar{\gb}^{n-i}=x\right\},
 \end{align*}
  for $x\in\bar{Z}$. It is immediate that $\bar{\mu}_n(\bar{Z}\backslash{\bar{X}(\beta)})=0$, that $\bar{\mu}_n(x)=\mu_n(x_1)$ and $\bar{\mathcal N}_n(x)=\mathcal N_n(x_1)$. We set 
  \[\pi_c(x_1,\cdots ,x_{d+s})=(x_{d+1},...,x_{d+s})\]  to be the projection onto the contracting directions, and  $S_i:=(\pi_c\circ T_i)|_{\mathbb K_c}$. The maps $S_i$ are contractions.  Let $\mathcal{R}$ be the attractor of the overlapping iterated function scheme $\{S_{-1},S_{0},S_{1}\}$. We have immediately that
\begin{align*}
\pi_c(\bar{X}(\beta))&=\pi_c\left\{T_{a_n}\circ...\circ T_{a_1}(\underline 0):n\in\mathbb{N},a_1...,a_n\in\{-1,0,1\}\right\}\\
&=\left\{S_{a_n}\circ...\circ S_{a_1}(\underline 0):n\in\mathbb{N},a_1...,a_n\in\{-1,0,1\}\right\}\subset\mathcal R
\end{align*}
since $0\in\mathcal R$.

  	\begin{definition}\label{CylinderDef}
	Let $a=(a_1,...,a_n)\in\{-1,0,1\}^n$. We define $[a]:=S_{a_1}\circ...\circ S_{a_n}(\mathcal{R})$. 	
	\end{definition}

Finally we define a set of small differences between points in $\overline X(\beta)$. 
\begin{definition}\label{DeltaDef} Let
\begin{align*}
\Delta=\{x-y: &x,y\in \overline X(\beta) \text{ and } \\
&\exists c_1\cdots c_n, d_1\cdots d_n\in\{-1,0,1\}^n: T_{c_n}\circ \cdots T_{c_1}(x)=T_{d_n}\cdots T_{d_1}(y)\}.
\end{align*}
\end{definition}
That is, $\Delta$ is the set of differences between points $x,y\in \overline X(\beta)$ which can be mapped to the same point in the future by the application of maps $T_i$. $\Delta$ is finite, we write $\Delta=\{v_1,\cdots,v_k\}$ with $v_1=0$.

  In this section we prove Theorems \ref{Thm1}, \ref{Thm2} and \ref{Thm3} by proving higher dimensional analogues. In particular, in subsection \ref{sub1} we prove that, for some  $\gl>0$, the measure $\frac{\bar{\mu}_n}{\gl^n}$ converges to an infinite stationary measure $\bar{\mu}$ (Proposition \ref{measure exists}, which has Theorem \ref{Thm1} as a direct corollary. 
  
 In subsection \ref{sub3} we define matrices $A_{-1},A_{0},A_{1}$ playing the role of $M_{-1},M_{0},M_{1}$ of the Golden mean example. Given a point $x=T_{a_n}\circ...\circ T_{a_1}(0)$, where $a_i\in\{-1,0,1\}$, we use the matrix $A_{a_1}\cdot...\cdot A_{a_n}$ to compute the measure $\bar{\mu}$ locally around $x$ (Proposition \ref{prop2}),  which has Theorem \ref{Thm2} as a direct corollary. 
 
Finally in subsection \ref{LMCS} we show that information about the position of $\pi_c(x)$ determines the last few elements $a_{\gk},...,a_{n}$ of a code of $x$. This allow us to use arguments involving a modified Birkhoff metric on the product $A_{a_1}\cdot...\cdot A_{a_n}$ to estimate the local measure around $x$ based on information about $\pi_c(x)$. This gives rise to Proposition \ref{mainpro}, which has Theorem \ref{Thm3} as a corollary, as explained directly after the proof of Proposition \ref{mainpro}.

	\subsection{The limit measure $\bar{\mu}$}\label{sub1}
	




   We will denote the vector space of signed measures on $\bar{Z}$ by $\mathcal{M}(\bar{Z})$. For $\nu\in\mathcal{M}(\bar{Z})$ we set
 \begin{align*}
 ||\nu||=\sum_{x\in \bar{Z}}|v(x)|.
 \end{align*}
There is a recursive way to go from $\bar{\mu}_n$ to $\bar{\mu}_{n+1}$ which gives a dynamical description of $\bar{\mu}_n$. 
\begin{align*}
&\bar{\mu}_{n+1}(x)=\#\left\{(a_1,...,a_{n+1},b_1,...,b_{n+1})\in\{0,1\}^{2n}:\sum_{i=1}^{n+1}a_i\bar{\gb}^{n+1-i}-\sum_{i=1}^{n+1}b_i\bar{\gb}^{n+1-i}=x\right\}
\\=&\#\left\{(a_1,...,a_{n+1},b_1,...,b_{n+1})\in\{0,1\}^{2(n+1)}:T_{a_{n+1}-b_{n+1}}\left(\sum_{i=1}^{n}a_i\gb^{n-i}-\sum_{i=1}^{n}b_i\gb^{n-i}\right)=x\right\}
\\=&\sum_{(a,b)\in\{0,1\}^2}\#\left\{(a_1,...,a_n,b_1,...,b_n)\in\{0,1\}^{2n}:\sum_{i=1}^{n-1}a_i\gb^{n-i}-\sum_{i=1}^{n-1}b_i\gb^{n-i}=T_{a-b}^{-1}(x)\right\}
\\=&\sum_{(a,b)\in\{0,1\}^2}\bar{\mu}_{n}(T_{a-b}^{-1}(x)).
\end{align*}

\begin{definition}
We define the operator $L$ on $\mathcal M(Z)$ by letting
\begin{align*}
(L(\nu))(A):=\sum_{(a,b)\in\{0,1\}^2}\nu(T_{a-b}^{-1}(A)).
\end{align*}
for $A\subset Z$.
\end{definition}

Then $\bar{\mu}_n$ satisfies  
\begin{equation*}
\bar{\mu}_n=L^n\bar{\mu}_0.
\end{equation*}

\begin{lemma}\label{maximality of center}
	For all $n\in\mathbb{N}$ and $y\in\bar{X}(\beta)$ we have $\bar{\mu}_n(y)\leq\bar{\mu}_n(0)$.	
\end{lemma}

\begin{proof}

This follows from the Cauchy-Schwarz inequality. Define 
\begin{equation*}
\mu'_n(x)=\#\left\{a_1,...,a_n\in\{0,1\}^n:\sum_{i=1}^na_i\bar{\gb}^{n-i}=x \right\}
\end{equation*}
By the construction of $\bar{\mu}_n$ we have that
\begin{align*}
\bar{\mu}_n(y)&=\sum_{x\in \bar{Z}}\mu'_n(x)\mu'_n(x+y)
\\&\leq\left(\sum_{x\in  \bar{Z}}\mu'_n(x)^2\right)^{1/2}
\left(\sum_{x\in  \bar{Z}}\mu'_n(x+y)^2\right)^{1/2}
\\&\leq\left(\sum_{x\in  \bar{Z}}\mu'_n(x)^2\right)^{1/2}
\left(\sum_{x\in  \bar{Z}}\mu'_n(x)^2\right)^{1/2}
\\&=\sum_{x\in  \bar{Z}}\mu'_n(x)^2
\\&=\sum_{x\in  \bar{Z}}\mu'_n(x)\mu'_n(x)
\\&=\bar{\mu}_n(0)
\end{align*}
\end{proof}


Now we prove that the measure $\bar{\mu}$ exists. To do this, we show that it exists on arbitrarily large neighbourhoods of the origin. Let

\[
  I_{\gb_i}(R)=\left\lbrace\begin{array}{cc}
    \left(\frac{-R}{||\gb_i|-1|},\frac{R}{||\gb_i|-1|}\right),& \quad\gb_i\in\mathbb{R}\backslash\{-1,1\}\\
   \left\{z\in\mathbb{C}:|z|<\frac{R}{||\gb_i|-1|}\right\},& \quad\gb_i\in\left\{z\in\mathbb{C}:|z|\neq1\right\}\backslash\mathbb{R}\end{array}\right. ,
\]
$B_{\beta}(R)=\Pi_{i=1}^{d+s}I_{\gb_i}(R)$, and  
\begin{align*}
\bar{X}_R(\beta)&:=\bar{X}(\beta)\cap B_{\gb}(R).
\end{align*}

Observe that \[T_i(\bar{X}(\beta)\backslash \bar{X}_R(\beta))\subseteq \bar{X}(\beta)\backslash \bar{X}_R(\beta)\] for $R\geq1$ and $i\in\{-1,0,1\}$. This means that, for $R>1$ and $x\in \bar{X}_R(\beta)$, any word $a_1\cdots a_n$ for which $T_{a_n}\circ \cdots T_{a_1}(0)=x$ has that all the intermediate orbit points $T_{a_m}\circ \cdots T_{a_1}(0)$ for $m<n$ also lie in $\bar{X}_R(\beta)$.  Thus, for $x\in \bar{X}_R(\beta)$ we can compute $\bar{\mathcal N}_n(x)$ just by studying the dynamics of the maps $T_i$ restricted to $\bar{X}_R(\beta)$.

Since $\bar{X}_R(\beta)$ is a bounded subset of a lattice, it is finite, we enumerate its elements $\{x_1,\cdots x_{k_R}\}$ with $x_1=0$. Then we write down the matrix
	\begin{align*}
\Lambda_R(i,j)=
\begin{cases}
1\quad \text{if }\quad T_1(x_i)=x_j \text{ or} T_{-1}(x_i)=x_j \\
2\quad \text{if }\quad T_0(x_i)=x_j\\
0\quad \text{otherwise}
\end{cases}.
\end{align*}
which encodes the dynamics on $\bar{X}_R(\beta)$ given by the maps $T_i$. Then since $\bar{\mathcal N}_n(x_j)$ counts the number of length $n$ orbit pieces from $0$ to $x_j$ under the maps $T_0, T_1, T_{-1}$, double counting for each use of $T_0$, we see that
\[
\bar{\mathcal N}_n(x_j)=(\Lambda_R^n)_{1,j}.
\]



From $T_i(\bar{X}(\beta)\backslash \bar{X}_1(\beta))\subset \bar{X}(\beta)\backslash \bar{X}_1(\beta)$  we get that the irreducible component of $\Lambda_R$ that contains the zero point is contained in $\bar{X}_1(\beta)$ so by lemma \ref{maximality of center} we have that the spectral radius of $\Lambda_R$ is equal to the spectral radius of $\Lambda_1$ for all $R>1$.

\begin{definition}

We set $\lambda:=\rho(\Lambda_1)$.

\end{definition}

Now if we knew that the matrices $\Lambda_R$ were irreducible, the existence of $\mu$ would be immediate. As it is we require the following lemma, the proof of which is postponed to the appendix.

\begin{lemma}\label{Non-irreducible Perron}
		Let $A$ be a non-negative $N\times N$ matrix and $e_1=(1,0,0,...,0)\in\mathbb{R}^N$. Assume that 
		
		\begin{itemize}

		 \item[i)] $A(1,1)>0$,
		 \item[ii)] there exists $n\in\mathbb{N}$ such that $e_1A^n$ is stricly positive,		
		
		 \item[iii)]$e_1A^n(i)\leq e_1A^n(1)$ for all $n\in\mathbb{N}$ and $i\in\{1,...,N\}$,
		 
		\end{itemize}

	then $\lim_{n\rightarrow\infty}e_1A^n/\rho(A)^n$ exists. 
	
		\end{lemma}

Now by the construction of $\Lambda_R$ and by Lemma \ref{maximality of center} and Lemma \ref{Non-irreducible Perron} we have the following proposition.

\begin{proposition}\label{measure exists}

For each $x\in\bar{X}(\beta)$ 
\begin{align*}
\bar{\mu}(x):=\lim_{n\rightarrow\infty}\frac{\bar{\mathcal N}_n(x)}{\lambda^n}
\end{align*}
exists, defining a measure $\bar{\mu}\in\mathcal{M}(\bar{Z})$. 

\end{proposition}
We conclude this section with three lemmas showing that the measure $\mu$ is invariant under $L$, that $\lambda<4$, and that the total mass of the measure $\mu$ is infinite.

\begin{lemma}\label{mu stationary}
	$L\bar{\mu}=\gl\bar{\mu}$
\end{lemma}

\begin{proof}
	
	For all $x\in\bar{X}(\beta)$ we have 
	\begin{align*}	
L\bar{\mu}(x)&=\bar{\mu}(T_{-1}^{-1}(x))+2\bar{\mu}(T_{0}^{-1}(x))+\bar{\mu}(T_{1}^{-1}(x))\\
&=\lim_{n\rightarrow\infty}\frac{1}{\gl^n}\left(\bar{\mu}_n(T_{-1}^{-1}(x))+2\bar{\mu}_n( T_{0}^{-1}(x))+\bar{\mu}_n(T_{1}^{-1}(x))\right)\\
&=\lim_{n\rightarrow\infty}\frac{1}{\gl^n}L\bar{\mu}_n(x)\\
&=\gl\lim_{n\rightarrow\infty}\frac{1}{\gl^{n+1}}\bar{\mu}_{n+1}(x)\\
&=\gl\bar{\mu}(x).
\end{align*}
\end{proof}

For sets $X$, measures $\nu\in\mathcal M(X)$ and measurable sets $A\subset X$ we let $\nu|_A$ be such that $\nu|_A(B)=\nu(A\cap B)$ for all measurable $B\subset X$. 

\begin{lemma}
	$\gl<4$
\end{lemma}

\begin{proof}
It is clear that if $\nu\in\mathcal{M}(\bar{Z})$ is such that 

\begin{equation*}
||\nu||<\infty
\end{equation*}

then 

\begin{equation*}
||L\nu||=4||\nu||.
\end{equation*}

 
 
Note that $L(\bar{\mu})=\lambda\bar{\mu}$ and \[
L\left(\bar{\mu}|_{\bar{X}_1(\beta)}\right)|_{ {\bar{X}_1(\beta)}}=\lambda\bar{\mu}|_{X_1(\beta)},
\]
but
\[
\left|\left|\left(L\left(\bar{\mu}|_{\bar{X}_1(\beta)}\right)\right)|_{ \bar{Z}\backslash {\bar{X}_1(\beta)}}\right|\right|>0
\]
since $\bar{X}_1(\beta)$ is not invariant under the maps $T_0, T_1, T_{-1}$. Then
\begin{align*}
4||\bar{\mu}|_{\bar{X}_1(\beta)}||&=\left|\left|L\left(\bar{\mu}|_{\bar{X}_1(\beta)}\right)\right|\right|
=\left|\left|\left(L\left(\bar{\mu}|_{\bar{X}_1(\beta)}\right)\right)|_{ {\bar{X}_1(\beta)}}\right|\right|+\left|\left|\left(L\left(\bar{\mu}|_{\bar{X}_1(\beta)}\right)\right)|_{ \bar{Z}\backslash {\bar{X}_1(\beta)}}\right|\right|\\
&=\gl||\bar{\mu}|_{\bar{X}_1(\beta)}||+\left|\left|\left(L\left(\bar{\mu}|_{\bar{X}_1(\beta)}\right)\right)|_{ \bar{Z}\backslash {\bar{X}_1(\beta)}}\right|\right|>\gl||\bar{\mu}|_{\bar{X}_1(\beta)}||_1
\end{align*}
giving us $\gl<4$.
\end{proof}

\begin{proposition}
	$||\bar{\mu}||=\infty$, i.e., the measure $\bar{\mu}$ is infinite.
\end{proposition}

\begin{proof}
	 For $n\in\mathbb{N}$ we get
	\begin{align*}
	||\bar{\mu}||&=\left|\left|\frac{1}{\gl^n}L^n\bar{\mu}\right|\right|>\left|\left|\frac{1}{\gl^n}L^n\left(\bar{\mu}|_{\{0\}}\right)\right|\right|
	=\frac{4^n}{\gl^n}\bar{\mu}(0).
	\end{align*}
	The result follows since $\gl<4$, $\bar{\mu}(0)>0$ and $n$ was arbitrary. 
	
		\end{proof}

\subsection{Transition Matrices}\label{sub3}
	Let $ \Delta=\{v_1,\cdots,v_k\}$ with $v_1=0$. We introduce a $k\times k$ matrix with rows/columns corresponding to the points in $\Delta$.
	\begin{definition}\label{Ai}
		For $i\in\{-1,0,1\}$ let $A_i$ be the $k\times k$ matrix such that 
		\begin{align*}
		(A_{i})_{m,n}=
		\begin{cases}
		1\quad \text{if }\quad\exists j\in\{-1,1\}:T_{j-i}(v_m)=v_n \\
		2\quad \text{if }\quad T_{-i}(v_m)=v_n\\
		0\quad \text{otherwise}
		\end{cases}.
		\end{align*}
	\end{definition}
The matrices $A_i$ describe the evolution of local measure as we move from $x$ to $T_i(x)$, as described in Lemma \ref{AiTransitions}. Recall that $v_1=0, v_2, \cdots v_k$ are the elements of $\Delta$ (Definition \ref{DeltaDef}. We define a vector which describes the local measure around $x$.
\begin{definition}
We let $v(x)=(\mu(x), \mu(x+v_2), \cdots ,\mu(x+v_k))$.
\end{definition}

\begin{lemma}\label{AiTransitions}
Let $x\in \bar{X}(\beta)$. Then
\[
\frac{1}{\lambda}v(x)A_i=v(T_i(x)).
\]
\end{lemma}	
\begin{proof}
We show that
\[
(\bar{\mathcal N}_n(x), \bar{\mathcal N}_n(x+v_2), \cdots, \bar{\mathcal N}_n(x+v_k))A_i=(\bar{\mathcal N}_{n+1}(T_i(x)), \bar{\mathcal N}_{n+1}(T_i(x)+v_2), \cdots, \bar{\mathcal N}_{n+1}(T_i(x)+v_k)),
\]
the result will follow from this statement. 

Note that 
\begin{equation}\label{DeltaEvolution}
\bar{\mathcal N}_{n+1}(T_i(x)+v_l)=\bar{\mathcal N}_n(T_{1}^{-1}(T_i(x)+v_l))+\bar{\mathcal N}_n(T_{-1}^{-1}(T_i(x)+v_l))+2\bar{\mathcal N}_n(T_{0}^{-1}(T_i(x)+v_l))
\end{equation}
where of course $\bar{\mathcal N}_n(y)=0$ for $y\not\in\bar{X}(\beta)$. 

Secondly we note that 
\begin{align*}T_j(x+v_m)&=T_j(x)+T_0(v_m)\\
&= T_i(x)+T_0(v_m)+j-i\\
&= T_i(x)+T_{j-i}(v_m),
\end{align*}
which is equal to $T_i(x)+v_l$ if and only if $T_{j-i}(v_m)=v_l$.

So we can rewrite equation \ref{DeltaEvolution} to get
\begin{align*}
\bar{\mathcal N}_{n+1}(T_i(x)+v_l)&= \sum_{m\in\{1,\cdots,k\}}\bar{\mathcal N}_n(x+v_m)\chi_{T_{1-i}(v_m)=v_l}\\
&+\sum_{m\in\{1,\cdots,k\}}\bar{\mathcal N}_n(x+v_m)\chi_{T_{-1-i}(v_m)=v_l}\\
&+2\sum_{m\in\{1,\cdots,k\}}\bar{\mathcal N}_n(x+v_m)\chi_{T_{-i}(v_m)=v_l}.
\end{align*}
which is precisely the $l$th entry of $(\bar{\mathcal N}_n(x), \bar{\mathcal N}_n(x+v_2), \cdots, \bar{\mathcal N}_n(x+v_k))A_i$.






\end{proof}

\begin{prop}\label{prop2}
Set $W=v(0)=(\mu(0), \mu(v_2), \cdots, \mu(v_k))$. Let $x=\sum_{i=1}^n c_i\beta^{n-i}$. Then 
\[
v(x)=\frac{1}{\lambda^n}(WA_{c_1}\cdots A_{c_n}).
\]
In particular, 
\[
\bar{\mu}(x)=\frac{1}{\lambda^n}(WA_{c_1}\cdots A_{c_n})_1,
\]
i.e. the first entry of the $1\times k$ vector $\frac{1}{\lambda^n}WA_{c_1}\cdots A_{c_n}$. 
\end{prop}
\begin{proof}
This follows immediately from the previous lemma by writing
\[
x=T_{a_n}\circ T_{a_{n-1}}\circ\cdots \circ T_{a_1}(0).
\]
\end{proof}
Since the one dimensional measure $\mu$ is the projection of $\bar{\mu}$ onto the first coordinate, Theorem \ref{Thm2} follows as a direct corollary to Proposition \ref{prop2}.

	\subsection{Approximating local measures via the contractive subspace}\label{LMCS}

     Recall that $\mathcal{R}$ is the attractor of the IFS $\{S_{-1},S_{0},S_{1}\}$ and that $\pi_c(\bar{X}(\beta))\subseteq\mathcal{R}$. We will assume the following condition.

	\begin{condition}\label{condition 1}
	$\bar{X}(\beta)\cap cl(B_{\beta}(1))=\bar{Z}\cap\pi_c^{-1}(\mathcal{R}^{\mathrm{o}})\cap cl(B_{\beta}(1))$
	
	\end{condition}
	This is similar to a condition appearing in Corollary 4.5 of \cite{HareMasakovaVavra}. Here $\mathcal R^{\mathrm{o}}$ denotes the interior of the set. Condition \ref{condition 1} is a condition about two finite sets being equal, and so can be easily checked. In words, the condition says that a finite patch around zero of the set $\bar{X}(\beta)$, which is a higher dimensional analogue of the spectrum of $\beta$, can be written as a patch of a cut and project set with window $\mathcal R^{\mathrm{o}}$. Condition \ref{condition 1} implies that the whole set $\bar{X}(\beta)$ can be written as a cut and project set, this is the content of Corollary \ref{cutprojectcor}. In every example we have checked with $\beta\in(1,2)$ a hyperbolic algebraic unit and alphabet $\mathcal A=\{-1,0,1\}$, Condition \ref{condition 1} does indeed hold, but there are examples of Hare, Mas\'akov\'a and V\'avra \cite{HareMasakovaVavra} using complex alphabets in which the cut and project set contains extra points. 
	
	\begin{lemma}\label{closedXbar}
		For each $i\in\{-1,0,1\}$ we have $T_i^{-1}(\bar{Z})\subseteq \bar{Z}$.
	\end{lemma}

\begin{proof}
We need only show that for $x=\sum_{i=0}^{d-1} z_i\beta^i$ where $z_0,\cdots, z_{d-1}\in \mathbb Z$ we have that there exist $z_0', \cdots z_{d-1}'$ such that $\frac{x}{\beta}=\sum_{i=0}^{d-1}z_i'\beta^i$. Once we have shown this for x, the corresponding results for the Galois conjugates follow directly.

The result holds because, for $\beta$ to be a root of a $\{-1,0,1\}$-polynomial, it is necessary that the final term $a_0$ of the minimal polynomial\footnote{The fact that $\beta$ is a root of a $\{-1,0,1\}$-polynomial isn't enough to imply that the {\it minimal} polynomial of $\beta$ has digits only $\{-1,0,1\}$, but it does follow that the largest and smallest terms in the minimal polynomial are $\pm1$.} of $\beta$ is $\pm1$. Then we use
\begin{align*}
0&=a_d\beta^d+a_{d_1}\beta^{d-1}+\cdots + a_1\beta+a_0\\
\implies \frac{1}{\beta}&=\frac{a_d}{-a_0}\beta^{d-1}+\cdots + \frac{a_1}{-a_0}.
\end{align*}

and since each of the terms $\frac{a_i}{-a_0}$ are integers, since $a_0=\pm 1$, we have that dividing by $\beta$ keeps numbers within the integer lattice as required.
\end{proof}

\begin{proposition}\label{cylinder gives code Pisot}

Suppose that $x\in \bar{X}(\beta)$ has $\pi_c(x)\in[\e_1,...,\e_n]^{\mathrm{o}}$ for some $\e_1,...,\e_n\in\{-1,0,1\}^n$. Then, under condition \ref{condition 1}, there are $a_1,...,a_\kappa\in\{-1,0,1\}$ such that 
	\begin{align*}
	T_{\e_1}\circ...\circ T_{\e_n}\circ T_{a_\kappa}\circ...\circ T_{a_1}(0)=x.
	\end{align*}	
\end{proposition}

Recall that $[\e_1,\cdots,\e_n]$ is a subset of $\mathcal R$ defined in Definition \ref{CylinderDef}, and that $[\e_1,\cdots,\e_n]^{\mathrm{o}}$ is its interior.

\begin{proof}
	By the iterated function system construction of $\mathcal R$, the fact that $\pi_c(x)\in[\e_1,\cdots,\e_n]$ gives the existence of arbitrarily long words $a_1,\cdots a_m\in\{-1,0,1\}^m$ such that 	
	\begin{align*}
	\pi_c(x)\in S_{\e_1}\circ...\circ S_{\e_n}\circ S_{a_1}\circ...\circ S_{a_m}(\mathcal R).
	\end{align*} 	
	This implies that there is $y\in \bar{Z}$ with $\pi_c(y)\in\mathcal R$ such that 
	\begin{align*}
	x=T_{\e_1}\circ...\circ T_{\e_n}\circ T_{a_1}\circ...\circ T_{a_m}(y),
	\end{align*}
	the fact that $y\in \bar{Z}$ follows using Lemma \ref{closedXbar} using that $x\in \bar{Z}$. Now $x=(x_1\cdots,x_d,x_{d+1},\cdots x_{d+s})$ where the maps $T_i$ are expanding on the first $d$ coordinates and contracting on the final $s$ coordinates. Hence the maps $T_i^{-1}$ contract the first $d$ coordinates and for any $\epsilon>0$, for large enough $m$, the point
	\[
	y=(T_{\e_1}\circ...\circ T_{\e_n}\circ T_{a_1}\circ...\circ T_{a_m})^{-1}(x)
	\]
must have its first $d$ coordinates within distance $\epsilon$ of the box $\Pi_{i=1}^d I_{\beta_i}(1)$. But since these points lie in a uniformly discrete set, the first $d$ coordinates must actually lie in the closure of this box.

The final $s$ coordinates must be in $\mathcal R^{\mathrm{o}}$, since $\pi_c(x)\in S_{\e_1}\circ...\circ S_{\e_n}\circ S_{a_1}\circ...\circ S_{a_m}(\mathcal R^{\mathrm{o}})$. Thus
\[
(T_{\e_1}\circ...\circ T_{\e_n}\circ T_{a_1}\circ...\circ T_{a_m})^{-1}(x)\in \overline Z\cap \pi_c^{-1}(\mathcal R)\cap B_{\beta}(1),
\]
and so by Condition \ref{condition 1} there exists $b_1\cdots b_k\in\{-1,0,1\}^k$ such that
\[
(T_{\e_1}\circ...\circ T_{\e_n}\circ T_{a_1}\circ...\circ T_{a_m})^{-1}(x)=T_{b_1}\circ\cdots \circ T_{b_k}(0)\in \bar{X}_1(\beta).
\]
Then
\[
x=T_{\e_1}\circ...\circ T_{\e_n}\circ T_{a_1}\circ...\circ T_{a_m}\circ T_{b_1}\circ \cdots T_{b_k}(0)
\]
as required.
	
	
	
	
	

\end{proof}

	\begin{corollary}\label{cutprojectcor}
		Under condition \ref{condition 1}, $\bar{X}(\beta)=\bar{Z}\cap\pi_c^{-1}(\mathcal{R}^{\mathrm{o}})$.
	\end{corollary}	
This is just the statement of the previous proposition with $\e_1,\cdots \e_n$ being the empty word. A similar statement appears as Corollary 4.5 in \cite{HareMasakovaVavra}.

	\begin{lemma}\label{Matrices and dynamics}
		Let $i,j \in\{1,\cdots,k\}$. Then 
		 there exists $c_1,...,c_n\in\{-1,0,1\}$ such that		
		\begin{align*}
		(A_{c_1}\cdot...\cdot A_{c_n})_{ij}>0.
		\end{align*} 
	\end{lemma}
	
	\begin{proof}
The definition of $\Delta$ means there exist $a_1\cdots a_m\in\{-2,-1,0,1,2\}^m$ and $a_{m+1}\cdots a_n\in\{-2,-1,0,1,2\}$ such that $T_{a_m}\circ \cdots \circ T_{a_1}(v_i)=0$ and $T_{a_{m+1}}\circ \cdots \circ T_{a_n}(0)=v_j$. Then choosing $c_1\cdots c_m$ such that $a_i-c_i\in\{-1,0,1\}$ for each $i$ the result follows directly from the definition of $A_i$. 
	\end{proof}
	
The following lemma is important in defining for us a `mixing word' $a_n\cdots a_1\in\{-1,0,1\}^n$.
	
	\begin{proposition}\label{magic wordPisot}
		There is a word $w=w_1,...,w_n\in\{-1,0,1\}^n$ and $I,J\subseteq\Delta$ such that $0\not\in I, 0\not\in J$ and $(A_{w_1}\cdot...\cdot A_{w_n})_{i,j}=0\Leftrightarrow i\in I\text{ or } j\in J$.
	\end{proposition}
	\begin{proof}
We start by building a set $I$ and a word $w_1,\cdots, w_m$ such that the $i$th row of $A_{w_1}.\cdots.A_{w_m}$ is a zero row for $i\in I$ and $(A_{w_1}.\cdots.A_{w_m})_{i,1}>0$ otherwise.

		
		Step 1: Note that for $i\in\{-1,0,1\}$, $(A_i)_{1,1}>0$.
		
		Step 2: The point $v_2$ is in $\Delta$, and from the definition of $\Delta$ and lemma \ref{Matrices and dynamics} there exist $w_1\cdots w_{m_1}\in\{-1,0,1\}$ such that 
		\begin{align*}
		 (A_{w_1}\cdots A_{w_{m_1}})_{2,1}>0
		\end{align*}
		
		Step 3: Either the $3$rd row of the product $A_{w_1}\cdots A_{w_{m_1}}$ is a zero row, in which case we declare $v_3\in I$, or there exists $v_{p}\in \Delta$ with $(A_{w_1}\cdots A_{w_{m_1}})_{3,p}>0$. As in step 2, since $v_p\in \Delta$ choose a word $w_{m_1+1}\cdots w_{m_2}$ such that
		\[
		(A_{w_{m_1+1}}\cdots A_{w_{m_2}})_{p,1}>0.
		\]
		Then the product of matrices $A_{w_1}\cdots A_{w_{m_2}}$ has that entry $(3,1)$ is positive. Furthermore, entry $(2,1)$ is still positive, since $A_{w_1}\cdots A_{w_{m_1}}$ had entry $(2,1)$ positive, and then we are post multiplying by matrices with positive top left entry.
		
		Iterating this procedure, we create a word $w_1\cdots w_{m_k}$ and a set $I\subset \Delta$ such that the $i$th row of $A_{w_1}.\cdots.A_{w_{m_k}}$ is a zero row for $i\in I$ and $(A_{w_1}.\cdots.A_{w_{m_k}})_{i,1}>0$ otherwise.

Note that the matrices $A_1^T, A_0^T, A_{-1}^T$ also have top left entry strictly positive and that for any $i\in\{1,\cdots k\}$ there exists a word $c_1\cdots c_n$ such that $(A_{c_1}\cdots A_{c_n})_{(i,1)}>0$. So we repeat the above procedure for the matrices $A_1^T, A_0^T, A_{-1}^T$ to create a word $w_1'\cdots w_{n_k}'$ and a set $J$ such that the $j$th row of $A^T_{w_1'}\cdots A^T_{w_{n_k}'}$ is a zero row for $j\in J$, and $(A^T_{w_1'}\cdots A^T_{w_{n_k}'})_{(j,1)}>0$ otherwise.  

Taking the transpose once more gives us that the product $A_{w_{n_k}'}\cdots A_{w_1'}$ has a set $J$ of zero columns, and for all other columns the first entry is strictly positive.

Now setting $w_1\cdots w_n=w_1\cdots w_{m_k}w_{n_k}'\cdots w_1'$ we see that the product $A_{w_1}\cdots A_{w_n}$ has a set $I$ of zero rows, a set $J$ of zero columns, with all other entries strictly positive as required.

	\end{proof}

\begin{definition}
	Let the mixing word $w=w_1,...,w_n$ and $A_w=A_{w_1}\cdot...\cdot A_{w_n}$ where $w_1,...,w_n$ are as in Proposition \ref{magic wordPisot}
\end{definition}

		
	
Recall that we defined the $1\times k$ vectors
\[v(x)=(\mu(x), \mu(x+v_2),\cdots ,\mu(x+v_k))\] 
where $\Delta=(v_1,\cdots,v_k)$ with $v_1=0$. Map the space of $1\times k$ vectors with positive first entry onto projective space by letting $(V')_i=\frac{(V)_{i+1}}{(V)_1}$ for $1\leq i \leq 16$, giving
\[
v'(x)=\left(\frac{\mu(x+v_2)}{\mu(x)}, \frac{\mu(x+v_3)}{\mu(x)},\cdots \frac{\mu(x+v_k)}{\mu(x)}\right)
\]
As before, define the projective distance by
\[
d(U,V)=\max_{i\in\{1,\cdots,k-1\}} |\ln((V)_i)-\ln((U)_i)|\in[0,\infty].
\]
Here $\ln(0)-\ln(0)$ should be understood to take value $0$.
\begin{prop}\label{ConPro}
There exist $C_1>0$ and $C_2\in(0,1)$ such that for any $1\times k$ vectors $U,V$,
\begin{itemize}
\item $d(UA_w,VA_w)<C_1$
\item if $d(U,V)<\infty$ then $d(UA_w,VA_w)<C_2d(U,V).$
\item if $d(U,V)<\infty$ then $d(UA_i,vA_i)<d(U,V)$ for any $i\in\{-1,0,1\}$.
\end{itemize}
\end{prop}
If $A_w$ was a strictly positive matrix, this would be a standard result of Birkhoff \cite{GBirkhoff}. It is a simple modification to extend this to the matrices $A_w$, which are strictly positive on some block with all entries outside of this block zero. Details of this proof are given in the first author's thesis.

	\begin{proposition}\label{mainpro}
	Assume that Condition \ref{condition 1} holds. Then there exist positive constants $C_1, C_2$ such that for any word $a_1\cdots a_r\in\{-1,0,1\}^n$ and for any $x,y\in \bar X(\beta)$ with $\pi_c(x), \pi_c(y)\in[a]^{\mathrm{o}}$,
	\[
	d(v'(x),v'(y))<C_1C_2^{d(a)-1}
	\]
	where $d(a)$ is the number of disjoint occurences of $w$ in $a=a_1\cdots a_n$.
	\end{proposition}
	
	\begin{proof}
	By Lemma \ref{cylinder gives code Pisot} we have that $x$ and $y$ both have expansions ending with the word $a$, i.e. we can write $x=\sum_{i=1}^n c_i\beta^{n-i}$, $y=\sum_{i=1}^m d_i\beta^{m-i}$ where both $c_1\cdots c_n$ and $d_1\cdots d_m$ end in word $a_r\cdots a_1$. 
	
Then by Lemma \ref{AiTransitions} we can write
\[
v(x)=\frac{1}{\lambda^n}v(0)A_{c_1}\cdots A_{c_n}=\underbrace{\frac{1}{\lambda^n}v_0A_{c_1}\cdots A_{c_{n-r}}}_{:=U}A_{a_r}\cdots A_{a_1}
\]
and
\[
v(y)=\frac{1}{\lambda^n}v(0)A_{d_1}\cdots A_{d_m}=\underbrace{\frac{1}{\lambda^n}v_0A_{d_1}\cdots A_{d_{m-r}}}_{:=V}A_{a_r}\cdots A_{a_1}
\]  
But now $a_r\cdots a_1$ contains $d$ occurences of the mixing word $w$. the first of which contracts the distance between vectors $U$ and $V$ to at most $C_1$, and the final $d(a)-1$ of which each contract the distance by a factor of $C_2$, as in Proposition\ref{ConPro}. Then we have the required result.

	\end{proof}

We note that Theorem \ref{Thm3} follows as a direct corollary to Propsition \ref{mainpro}, as the vector $v'(x)$ can be written
\[
v'(x)=(\exp(f_2(x_c)),\exp(f_3(x_c)),\cdots \exp(f_k(x_c)))
\]
and that $d(v'(x),v'(y))<C_1C_2^{d(a)-1}$ implies that for each $i\in\{2,\cdots,k\}$ the differences $|\ln(f_i(x_c))-\ln(f_i(y_c)|< C_1C_2^{d(a)-1}$. Projecting $\bar{\mu}$ and the elements of $\Delta$ onto their first coordinates we are done.

Finally we show that all elements of $\bar X$ can be reached from $0$ by applying finitely many translations from the set $\Delta$.
\begin{lemma} 
Let $a_1,...a_m\in\{-1,0,1\}$ be such that $a_1\bar{\beta}^{m-1}+...+a^{m-1}\bar{\beta}+a_m\bar{\beta}^0=0$ and $a_1\neq0$. Then

 \begin{align*}
      \left\{\sum_{i=0}^\gk x_i: \gk\in\mathbb{N}, x_1,...,x_{\gk}\in\Delta\right\}=\bar{X}.
  \end{align*}
\end{lemma} 

\begin{proof}
Notice that $m\geq\deg(\gb)+1$. We have
\begin{align*}
    T_{a_m}\circ...\circ T_{a_1}(0)=0
\end{align*}
hence the set 
\begin{align*}
    B:&=\left\{T_{a_k}\circ...\circ T_{a_1}(0): 1 \leq k\leq m-1\right\}\\&=
    \left\{a_1\bar{\beta}^{k-1}+...+a^{k-1}\bar{\beta}+a_k\bar{\beta}^0: 1 \leq k\leq m-1\right\}
\end{align*}
is a subset of $\Delta$. Set
 \begin{align*}
      \Delta(0)=\left\{\sum_{i=0}^\gk x_i: \gk\in\mathbb{N}, x_1,...,x_{\gk}\in\Delta\right\}.
  \end{align*}
  
The proof is completed by showing inductively that $\bar{\beta}^0,...,\bar{\beta}^{m-1}\in\Delta(0)$. Indeed $\bar{\beta}^0\in B\subseteq\Delta$ and if $\bar{\beta}^0,...,\bar{\beta}^k\in\Delta(0)$, for $\kappa<m-1$, then 
\begin{align*}
\bar{\beta}^{k+1}=a_1((a_1\bar{\beta}^{k+1}+...+a^{k+1}\bar{\beta}+a_{k+2}\bar{\beta}^0) -a_2\bar{\beta}^{k}-...-a^{k}\bar{\beta}-a_{k+2}\bar{\beta}^0)\in\Delta(0).
\end{align*}

\end{proof}

\section{Appendix}\label{sub2}

	In this section we will prove Lemma \ref{Non-irreducible Perron}

	\begin{proof}
	
	By bringing the matrix to it's normal form of a reducible matrix, see (\cite{Varga2000}, p. 51), we can assume that 
	
	\begin{equation*}
	A=
	\begin{bmatrix}
	B_1& *& *&\cdots & *\\
	0& B_2 & * & \cdots & *\\
	\vdots&\vdots& \vdots & &\vdots\\
	0& 0 & 0 & \cdots & *\\
	0& 0 & 0 & \cdots & B_h
	\end{bmatrix}
	\end{equation*}
	
	where $B_i$ is a non-negative irreducible square matrix for $i\in\{1,...,h\}$. By rescaling we can assume that $\rho(A)=1$. Clearly $1=\rho(A)=\max\{\rho(B_1),...,\rho(B_h)\}$ so from assumption iii) we get $\rho(B_1)=1$. We set 
	\begin{equation*}
	S_i:=\left\{j\in\{1,...,N\}:\text{ The entry (j,j) is contained in the $B_i$-block }\right\}.
	\end{equation*}

	 For $i\in\{1,...h\}$ let 
	\begin{equation*}
	V_i:=\left\{u\in\mathbb{R}^N: u(j)=0\text{ if }j\notin S_i \right\}
	\end{equation*}
	and
	\begin{equation*}
	V_{i-}:=\left\{u\in\mathbb{R}^N: u(j)=0\text{ if }j\notin\cup_{\kappa=1}^{i-1}S_\kappa \right\}.
	\end{equation*}
	
	Define $p_i$ and $p_{i-}$ to be the orthogonal projections of $\mathbb{R}^N$ to the subspaces $V_i$ and $V_{i-}$ respectively. Finally let $B'_i$ to be $A$ where all entries outside the $B_i$-block are replaced by $0$ and $B'_{i-}$ to be $A$ where all the entries of the form $(i,j)$ are replaced by zero if and only if $j\notin\cup_{\kappa=1}^{i-1}S_\kappa$.
	
	We will prove the lemma by proving inductively that $p_i(e_1A^n)$ converges for $i\in\{1,...,h\}$. For $i=1$ we have that $p_i(e_1A^n)=p_i(e_1B'^n_1)$ so the statement is true since  $B_1$ is an irreducible aperiodic matrix of spectral radius one. The aperiodicity comes from assumption i). Now we assume that $i\in\{2,...,h\}$ and $p_{i-}(e_1A^n)$ converges to some $v'\in\mathbb{R}^N$ aiming to prove that $p_i(e_1A^n)$ converges.
	
	Case 1 $\rho(B_i)<1$: We define $T_i:\mathbb{R}^N\rightarrow \mathbb{R}^N$ by
	
	\begin{equation*}
	T_i(x)=xB_i'+p_i\left(v'A\right)
	\end{equation*}
	 
	 Since $\rho(B_i)<1$ there is $u'\in\mathbb{R}^N$ such that $u'(I-B_i')=p_i(v'A)$ so that

	\begin{equation*}
	T_i(x)=(x-u')B_i'+u'.
	\end{equation*}
	
	Now, from $\rho(B_i)<1$ again, we can conclude that $T_i^n(x)\rightarrow u'$ for any $x\in\mathbb{R}^N$. Writing 	
	\begin{equation*}
	p_i(e_1A^n)=T_i^n(0)+p_i(e_1A^n)-T_i^n(0)
	\end{equation*}	
	we only need to prove that $p_i(e_1A^n)-T_i^n(0)\rightarrow0$ to prove the convergence of  $p_i(e_1A^n)$ to $u'$. Let $\e>0$. By the spectral radius formula there exists $C>0$ such that 	
	\begin{equation*}
	||B_i'^n||\leq C\left(\rho(B_i)+\delta\right)^n
	\end{equation*}		
	where $\delta>0$ is chosen such that $\rho(B_i')+\delta<1$. Also by $p_{i-}(e_1A^n)\rightarrow v'$ we get that there is $\kappa_0$ such that  $|p_i(v'A)-p_i(p_{i-}(e_1A^{n-1})A)|<\e$. Notice that	
	\begin{align*}
	    p_i(e_1A^{\gk+1})=p_i(e_1A^{\gk})B_i'+p_i\left(p_{i-}(e_1A^\gk)\right),\quad \gk\in\{0,...\}.
	\end{align*}
	By iterating the relation above and choosing $n$ large enough we get
	\begin{align*}
	|p_i(e_1A^n)-T_i^n(0)|&=\left| \sum_{\kappa=1}^n\left(p_i(p_{i-}(e_1A^{\kappa-1})A)-p_i(v'A)\right)B_i'^{n-\kappa}\right|
	\\&\leq\left|\sum_{\kappa=1}^{\kappa_0-1}\left(p_i(p_{i-}(e_1A^{\kappa-1})A)-p_i(v'A)\right)B_i'^{n-\kappa}\right|+\sum_{\kappa=\kappa_0}^n||B_i^{n-\kappa}||\cdot\e
	\\&\leq \left|\left(\sum_{\kappa=1}^{\kappa_0-1}\left(p_i(v'A)-p_i(p_{i-1}(e_1A^{\kappa-1})A\right)B_i'^{\kappa_0-1-\gk}\right)B_i'^{n-\kappa_0+1}\right|
	\\&+\frac{\e \cdot C}{1-\rho(B_i)-\delta}
	\end{align*}
	Since $xB_i'^n\rightarrow0$ for all $x\in\mathbb{R}^N$ the above gives
	\begin{equation*}
	\limsup_{n\rightarrow\infty} 	|p_i^n(e_1A^n)-T_i^n(0)|\leq \frac{\e \cdot C}{1-\rho(B_i)-\delta}
	\end{equation*}	
	but since $\e$ was arbitrary we get 	
		\begin{equation*}
	\lim_{n\rightarrow\infty} 	|p_i^n(e_1A^n)-T_i^n(0)|=0
	\end{equation*}	
	completing the inductive step in the case $\rho(B_i)<1$.

	Case 2 $\rho(B_i)=1$: Now let $u'$ be a left eigenvector of $1$ of $B_i'$ with all entries in $S_i$ being positive. There exists such a $u'$ from Perron–Frobenius theorem since $B_i$ is a non-negative irreducible matrix. There are $\kappa_0$,$m\in\mathbb{N}$ and $c>0$ such that all entries in $S_i$ of	
	\begin{equation*}
	p_i\left(p_{i-}(e_1A^n)A^m\right)-cu'
	\end{equation*}	
	are positive for all $n>\gk_0$. This is true, by choosing $c$ small enough, because of assumption ii) and $p_{i-}(e_1A^n)\rightarrow v'$. Let $\kappa_1\in\mathbb{N}$ be such that $m(\gk_1-1)>\kappa_0$. The inequalities in the following are to be understood entrywise. For $n$ large enough we have, 
	\begin{align*}
	p_i(e_1A^{nm})&=\sum_{\kappa=1}^n\left(p_i\left(p_{i-}\left(e_1A^{m(\kappa-1)}\right)A^m\right)\right)B_i'^{m(n-\kappa)}
	\\&\geq\sum_{\kappa=\kappa_1}^n\left(p_i\left(p_{i-}\left(e_1A^{m(\kappa-1)}\right)A^m\right)\right)B_i'^{m(n-\kappa)}	
	\\&=\sum_{\kappa=\gk_1}^n\left(p_i\left(p_{i-}\left(e_1A^{m(\kappa-1)}\right)A^m\right)-cu'\right)B_i'^{m(n-\kappa)}+
	\sum_{\kappa=\gk_1}^ncu'B_i'^{m(n-\kappa)}
	\\&\geq\sum_{\kappa=\gk_1}^ncu'B_i'^{m(n-\kappa)}=(n-\gk_1+1)cu'.
	\end{align*}
	
	The above implies that $||p_i(e_1A^{nm})||_1\rightarrow\infty$ which contradicts assumption iii). Thus case 2 never occurs.

	\end{proof}

\section{Further Questions:}
We have a number of further questions on the structure of the sets $X(\beta)$, the measure $\mu$, and on how one can start to study $\mu$ using ergodic theory.

{\bf Question 1:} Is it the case for any integer alphabet $\mathcal A$ and for any hyperbolic $\beta$ one can express $X(\beta)$ (or the higher dimensional analogue $\tilde X(\beta)$ in the non-Pisot case) as a cut and project set with window $\mathcal R$ (or maybe $\mathcal R^{\mathrm{o}}$) defined as the attractor of an iterated function system $\{S_i:i\in\mathcal A\}$ where $S_i$ is defined in terms of the Galois conjugates of $\beta$ of absolute value less than one? We have shown an inclusion in Corollary \ref{cutprojectcor}. This question is also considered in \cite{HareMasakovaVavra}.

{\bf Question 2:} Is it true that, for a sequence of Pisot numbers $\beta_n$ of increasing degree in any interval $(1,2-\epsilon)$, the sequence of sets $\frac{1}{\beta_n-1}\left(X_{\{-1,0,1\}}(\beta_n)\cap\left[\frac{-1}{\beta-1},\frac{1}{\beta-1}\right]\right)$ equidistribute in $[-1,1]$? These sets are just pieces of the spectra of $X_{\{-1,0,1\}}(\beta_n)$ renormalised to live on $[-1,1]$. 

In Conjecture 2 we predict that, for such a sequence of Pisot numbers $\beta_n$, the distance between measures $\mu_{I_{\beta_n}}$ and normalised Lebesgue measure on $I_{\beta_n}$ tends to zero as $n$ tends to infinity. Our question here is the corresponding question for the sets supp$(\mu_{I_{\beta_n}})=X_{\{-1,0,1\}}(\beta_n)\cap\left[\frac{-1}{\beta-1},\frac{1}{\beta-1}\right]$. If the answer to Question 1 is positive, then this is a question about the structure of a sequence of cut and project sets. 

{\bf Question 3:} Does further numerical evidence support our Conjectures 1 and 2 on the dimension of Bernoulli convolutions and the distribution of measures $\mu_{I_{\beta_n}}$? The case that $\beta_n$ is a sequence of Pisot numbers converging to a limit in $(1,2)$ is of particular interest. In that case the limit must also be a Pisot number.

{\bf Question 4:} In the special case of the Golden mean, Theorem \ref{GMCocycle} describes how the measure $\mu$ evolves as one moves through the spectrum. Can one use this theorem, for example, to prove that the sequence of probability measures
\[
\lim_{n\to\infty}\frac{1}{\sum_{x\in X(\phi)\cap [0,n]} \mu\{x\}}\sum_{x\in X(\phi)\cap [0,n]} \mu\{x\}\delta_{x (mod 1)}
\]
converges weak$^*$ to Lebesgue measure on $[0,1]$? Inducing on the region $\{(x,y,z):y\in[0,\phi^2]\}$ we have an irrational rotation in the $x$ direction, and an irrational rotation in the $y$ direction which also gives the weights which tell us how to evolve the measure $\mu$. Then one might believe our question has a positive answer, since the weights $\mu(x)$ are driven by the evolution in the $y$ direction which is somehow independent of our position in the $x$ direction.

\section{Acknowledgements}
Tom Kempton is partially supported by EPSRC grant EP/T010835/1. We are grateful to Paul Mercat, Nikita Sidorov and Tom\'a\v s V\'avra for useful discussions.

\bibliographystyle{abbrv} 
\bibliography{CP}

\begin{thebibliography}{10}

\bibitem{AkiyamaTiling}
S.~Akiyama.
\newblock Self affine tiling and {P}isot numeration system.
\newblock In {\em Number theory and its applications ({K}yoto, 1997)}, volume~2
  of {\em Dev. Math.}, pages 7--17. Kluwer Acad. Publ., Dordrecht, 1999.

\bibitem{AFKP}
S.~Akiyama, D.-J. Feng, T.~Kempton, and T.~Persson.
\newblock {On the Hausdorff Dimension of Bernoulli Convolutions}.
\newblock {\em International Mathematics Research Notices}, 09 2018.
\newblock rny209.

\bibitem{AkiyamaKomornik}
S.~Akiyama and V.~Komornik.
\newblock Discrete spectra and {P}isot numbers.
\newblock {\em J. Number Theory}, 133(2):375--390, 2013.

\bibitem{GBirkhoff}
G.~Birkhoff.
\newblock Linear transformations with invariant cones.
\newblock {\em Amer. Math. Monthly}, 74:274--276, 1967.

\bibitem{VarjuBreuillard1}
E.~Breuillard and P.~P. Varj\'{u}.
\newblock Entropy of {B}ernoulli convolutions and uniform exponential growth
  for linear groups.
\newblock {\em J. Anal. Math.}, 140(2):443--481, 2020.

\bibitem{Bugeaud96}
Y.~Bugeaud.
\newblock On a property of {P}isot numbers and related questions.
\newblock {\em Acta Math. Hungar.}, 73(1-2):33--39, 1996.

\bibitem{EJK}
P.~Erd\H{o}s, I.~Jo\'{o}, and V.~Komornik.
\newblock On the sequence of numbers of the form
  {$\epsilon_0+\epsilon_1q+\cdots+\epsilon_nq^n,\ \epsilon_i\in\{0,1\}$}.
\newblock {\em Acta Arith.}, 83(3):201--210, 1998.

\bibitem{ErdosPisot}
P.~Erd{\H{o}}s.
\newblock On a family of symmetric {B}ernoulli convolutions.
\newblock {\em Amer. J. Math.}, 61:974--976, 1939.

\bibitem{FengTopology}
D.-J. Feng.
\newblock On the topology of polynomials with bounded integer coefficients.
\newblock {\em J. Eur. Math. Soc. (JEMS)}, 18(1):181--193, 2016.

\bibitem{FengWen}
D.-J. Feng and Z.-Y. Wen.
\newblock A property of {P}isot numbers.
\newblock {\em J. Number Theory}, 97(2):305--316, 2002.

\bibitem{GarsiaAC}
A.~M. Garsia.
\newblock Arithmetic properties of {B}ernoulli convolutions.
\newblock {\em Trans. Amer. Math. Soc.}, 102:409--432, 1962.

\bibitem{GarsiaEntropy}
A.~M. Garsia.
\newblock Entropy and singularity of infinite convolutions.
\newblock {\em Pacific J. Math.}, 13:1159--1169, 1963.

\bibitem{HKPS}
K.~G. Hare, T.~Kempton, T.~Persson, and N.~Sidorov.
\newblock Computing garsia entropy for bernoulli convolutions with algebraic
  parameters, 2019.

\bibitem{HareMasakovaVavra}
K.~G. Hare, Z.~Mas\'{a}kov\'{a}, and T.~V\'{a}vra.
\newblock On the spectra of {P}isot-cyclotomic numbers.
\newblock {\em Lett. Math. Phys.}, 108(7):1729--1756, 2018.

\bibitem{HS1}
K.~G. Hare and N.~Sidorov.
\newblock A lower bound for {G}arsia's entropy for certain {B}ernoulli
  convolutions.
\newblock {\em LMS J. Comput. Math.}, 13:130--143, 2010.

\bibitem{HS2}
K.~G. Hare and N.~Sidorov.
\newblock A lower bound for the dimension of {B}ernoulli convolutions.
\newblock {\em Exp. Math.}, 27(4):414--418, 2018.

\bibitem{HochmanInverse}
M.~{Hochman}.
\newblock {On self-similar sets with overlaps and inverse theorems for
  entropy}.
\newblock {\em Ann. of Math. (2)}, 140(2):773--822, 2014.

\bibitem{JessenWintner}
B.~Jessen and A.~Wintner.
\newblock Distribution functions and the {R}iemann zeta function.
\newblock {\em Trans. Amer. Math. Soc.}, 38(1):48--88, 1935.

\bibitem{CountingBeta}
T.~Kempton.
\newblock Counting {$\beta$}-expansions and the absolute continuity of
  {B}ernoulli convolutions.
\newblock {\em Monatsh. Math.}, 171(2):189--203, 2013.

\bibitem{KPV}
V.~Kleptsyn, M.~Pollicott, and P.~Vytnova.
\newblock Uniform lower bounds on the dimension of {B}ernoulli convolutions.
\newblock {\em Preprint}, 2021.

\bibitem{Lenz08}
D.~Lenz.
\newblock Aperiodic order and pure point diffraction.
\newblock {\em Philosophical Magazine}, 88(13-15):2059--2071, 2008.

\bibitem{RichardStrungaru}
C.~Richard and N.~Strungaru.
\newblock A short guide to pure point diffraction in cut-and-project sets.
\newblock {\em J. Phys. A}, 50(15):154003, 25, 2017.

\bibitem{Varga2000}
R.~S. Varga.
\newblock {\em Nonnegative Matrices}, pages 31--62.
\newblock Springer Berlin Heidelberg, Berlin, Heidelberg, 2000.

\bibitem{VarjuSummary}
P.~P. Varj\'{u}.
\newblock Recent progress on {B}ernoulli convolutions.
\newblock In {\em European {C}ongress of {M}athematics}, pages 847--867. Eur.
  Math. Soc., Z\"{u}rich, 2018.

\bibitem{VarjuTranscendental}
P.~P. Varj\'{u}.
\newblock On the dimension of {B}ernoulli convolutions for all transcendental
  parameters.
\newblock {\em Ann. of Math. (2)}, 189(3):1001--1011, 2019.

\end{thebibliography}

\end{document}